\documentclass{elsarticle} % [review] for line numbers
\journal{Computers and Mathematics with Applications}

\usepackage[cm]{fullpage}
\usepackage[]{lineno}
\usepackage{amsmath}
\usepackage{amssymb}
\usepackage{graphicx}
\usepackage[table]{xcolor}% http://ctan.org/pkg/xcolor
\usepackage{tabularx}
\usepackage{multirow}
\usepackage{makecell}
\usepackage{multicol}
\usepackage{float}
\usepackage{todonotes}
\usepackage{chngcntr}
\usepackage{multicol}
\usepackage{amsthm}
\usepackage{cleveref}
\usepackage{url}
\usepackage{mathtools}
\usepackage{thm-restate}
\usepackage{soul}
\counterwithin{table}{subsection}
\usepackage{algpseudocode}  % gets
\usepackage[utf8]{inputenc}
\newtheorem{theorem}{Theorem}
\newtheorem{lemma}{Lemma}
\newtheorem{assumption}{Assumption}
\newtheorem{remark}{Remark}

\usepackage[font=small,labelfont=bf]{caption}

\newcommand{\dual}{'}
\newcommand{\half}{{\frac 1 2}}
\newcommand{\reals}{{\mathbb{R}}}

\newcommand{\semi}[1]{\lvert#1\rvert}
\newcommand{\norm}[1]{\lVert#1\rVert}

\newcommand{\Rszi}[1]{R^{-1}_{#1}}
\newcommand{\mat}[1]{\mathsf{#1}}

\newcommand{\jump}[1]{\ensuremath{[\![#1]\!]} }
\newcommand{\avg}[1]{\ensuremath{\left\{\!\left\{#1\right\}\!\right\}} }

\newcommand{\nn}{\ensuremath{\mathbf{n}}}
\newcommand{\ff}{\ensuremath{\mathbf{f}}}
\newcommand{\ww}{\ensuremath{\mathbf{w}}}
\newcommand{\uf}{\ensuremath{\mathbf{u}_f}}

\newcommand{\vf}{\ensuremath{\mathbf{v}_f}}

\newcommand{\tG}{\ensuremath{T_t}}
\newcommand{\Tgn}{\ensuremath{\partial_{\nn}}}
\newcommand{\Tgne}{\ensuremath{\partial_{\nn,\epsilon}}}
\theoremstyle{definition}
\newtheorem{example}{Example}
\renewcommand{\AA}{\mathcal{A}}
\newcommand{\BB}{\mathcal{B}}
\numberwithin{example}{section}
\newcommand{\set}[1]{\{#1\}}
\newcommand{\dx}{\mathrm{d}x}
\newcommand{\ds}{\mathrm{d}s}

\usepackage{chngcntr}
\counterwithin{table}{section}

\newcommand{\ismuL}{{\frac 1 {\sqrt{\mu}}L^2(\Omega_f)}}
\newcommand{\smuH}{{\sqrt{\mu}H^1_{0, D}(\Omega_f) \cap \sqrt{D}L^2_{\tau}(\Gamma)}}

\newcommand{\sKH}{{\sqrt{K} H^1_{0, D}(\Omega_p)}}

\newcommand{\DSmulte}{{\sqrt{K}}L^2(\Gamma) \cap \frac 1 {\sqrt{\mu}} H^{-1/2}(\Gamma)}
\newcommand{\DSmultf}{\frac 1 {\sqrt{\mu}} H^{-1/2}(\Gamma)}
\newcommand{\DSmultpe}{{\sqrt{K}}L^2(\Gamma) }

\begin{document}

\begin{frontmatter}
\title{Robust preconditioning for coupled Stokes-Darcy problems with the Darcy problem in primal form}

\author[uio,suurph]{Karl Erik Holter}
\ead{karleh@math.uio.no}
\author[simula]{Miroslav Kuchta}
\ead{miroslav@simula.no}
\author[uio,simula]{Kent-Andre Mardal}
\ead{kent@math.uio.no}

\address[uio]{Department of Mathematics, Division of Mechanics, University of Oslo, Oslo, Norway}
\fntext[suurph]{Karl Erik Holter is a doctoral fellow in the Simula-UCSD-University of Oslo Research and PhD training (SUURPh) program, an international collaboration in computational biology and medicine funded by the Norwegian Ministry of Education and Research.}
\address[simula]{Department of Numerical Analysis and Scientific Computing, Simula Research Laboratory}

\begin{abstract}
The coupled Darcy-Stokes problem is widely used for modeling fluid transport in physical systems consisting of a porous part and a free part.
In this work we consider preconditioners for monolitic solution algorithms of the coupled Darcy-Stokes problem, 
where the Darcy problem is in primal form. 
We employ the operator preconditioning framework and utilize a fractional solver at the interface between the problems
to obtain order optimal schemes that are robust with respect to the material parameters, i.e. the permeability, viscosity and Beavers-Joseph-Saffman condition. 
Our approach is similar to that of \cite{paper1}, but since the Darcy problem is in primal form, the mass conservation at the interface introduces
some challenges. These challenges will be specifically addressed in this paper.
Numerical experiments illustrating the performance are provided. The preconditioner is posed in non-standard Sobolev spaces 
which may be perceived as an obstacle for its use in applications. However,  we detail the 
implementational aspects and show that the preconditioner is quite feasible to realize in practice.  
\end{abstract}
\end{frontmatter}

\section{Introduction}

Let $\Omega=\Omega_f\cup\Omega_p$, where $\Omega_f$ is the domain of the viscous flow, 
$\Omega_p$ is the domain of the porous media and $\Gamma$ their common interface.
Further let the domain boundaries be decomposed as $\partial \Omega_f = \Gamma \cup \partial \Omega_{f,D} \cup \partial \Omega_{f, N}$
and $\partial \Omega_p = \Gamma \cup \partial \Omega_{p,D} \cup \partial \Omega_{p, N}$, 
where subscripts $D, N$ signify respectively that Dirichlet and Neumann boundary conditions are prescribed on the part of the boundary. 
The boundary of $\Gamma$, i.e., the intersection of $\Gamma$ and $\partial\Omega$ is denoted by $\partial\Gamma$. 
An illustration is given in Figure  \ref{fig:DSdomain}.\\
\begin{minipage}{0.50\textwidth}
The Stokes problem reads: 
\begin{align}
\label{eq:stokes1} 
\mu \Delta \uf - \nabla p_f &= \ff  \text{ in } \Omega_f,\\
\label{eq:stokes2}
\nabla \cdot \uf &= 0 \text{ in } \Omega_f,  
\end{align}
while the Darcy problem in primal form reads:  
\begin{align}
      \label{eq:darcy} 
      -K \Delta p_p &= g  \text{ in } \Omega_p.
\end{align}
\null
\par\xdef\tpd{\the\prevdepth}
\end{minipage}
\begin{minipage}{0.50\textwidth}
  \begin{figure}[H]
    \includegraphics[width=1.0\columnwidth]{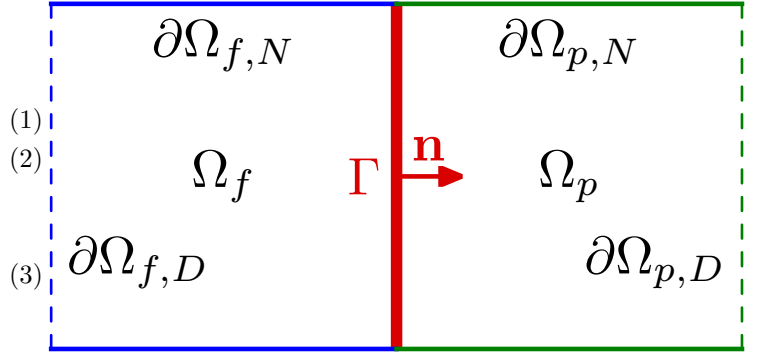}
    \vspace{-20pt}
    \caption{Schematic domain of Darcy-Stokes problem. Dirichlet conditions shown in dashed line, and interface in red.}
    \label{fig:DSdomain}    
    \vspace{5pt}    
\end{figure}
\end{minipage}
Here, $\uf, p_f$ are the unknown velocity and pressure for the Stokes problem \eqref{eq:stokes1}-\eqref{eq:stokes2} in $\Omega_f$,  $p_p$ is 
the unknown pressure of the Darcy problem \eqref{eq:darcy} in $\Omega_p$. The material parameters are the fluid viscosity $\mu$ and 
the permeability $K$. Here we shall consider the problem with the Dirichlet boundary conditions
\[
\uf = \uf^0\text{ on }\partial\Omega_{f, D},\quad p_p = p^0_p\text{ on }\partial\Omega_{p, D}
\]
and Neumann conditions
\[
\left(\mu\nabla\uf - p_f I\right)\cdot\mathbf{n}_f = \mathbf{h}\text{ on }\partial\Omega_{f, N},\quad \nabla p_p\cdot \mathbf{n}_p = h_p\text{ on }\partial\Omega_{p, N},
\]
where $\mathbf{n}_f$, $\mathbf{n}_p$ are the outer unit normals of the respective subdomains. 
In particular we assume that $\semi{\partial\Omega_{i, D}} > 0$ and $\semi{\partial\Omega_{i, N}} > 0$ for $i=p, f$.
Moreover, the coupled problem must be equipped with interface conditions expressing the continuity of stress as well as mass balance. We postpone their description until we describe the weak formulation of the problem.

%and $D = \frac {\mu \alpha_{\text{BJS}}} {\sqrt{\kappa}}$ is a coefficient of Beavers-Joseph-Saffman (BJS). At the interface $\Gamma$ the following conditions
%\eqref{eq:darcystokesmass}--\eqref{eq:darcystokesBJS} are respectively conservation of mass,
%balance of stress and the BJS condition \cite{mikelic2000interface}. We further assume that the problem is equipped with suitable boundary conditions.

The discretization of the coupled Darcy-Stokes problem with the Darcy problem in a mixed form is
challenging since the Darcy and Stokes problems, respectively, call for different schemes. For example, 
typical finite element methods for the Darcy problem, like the Raviart-Thomas or Brezzi-Douglas-Marini elements, are not stable for Stokes problem
as the discretization of the flux specifically targets the properties of $H(\mbox{div})$ rather than $H^1$ which is natural for Stokes discretizations.  
For this reason, a wide range of methods have been proposed over the last decade that address this particular challenge. For example, new
elements robust for both the Darcy and Stokes problem have been proposed in~\cite{arbogast2007computational,johnny2012family,karper2009unified,mardal2002robust,zhang2009low}. Alternatively, 
stabilization or modifications of standard methods may be used as in~\cite{burman2007unified,feng2010stabilized, xie2008uniformly}.
In this work we will consider the coupled problem with the Darcy equation in a primal form. 
Standard elements in both the Darcy and the Stokes domain will be used together with
a Lagrange multiplier to couple the unknowns appropriately at the interface.
%In particular,
%as we will see, the use of Lagrange multiplier enables
%the consideration of the Darcy problem in primal form interfacing the Stokes problem. 

The well-posedness of the Darcy-Stokes problem coupled together through the
use of a Lagrange multiplier is well-known  when the Darcy problem is in mixed form \cite{layton2002coupling,galvis2007non}, where both the continuous setting and various discretizations were proposed. Other solution and discretization algorithms for the coupled problem are presented in e.g. \cite{riviere2005locally, gatica2008conforming}, see \cite{discacciati2009navier,layton2002coupling} for an overview. For the mixed formulation we 
have, in our previous work \cite{paper1}, developed monolithic solvers that are 
robust with respect to all material parameters by utilizing fractional solvers on the interface.
Here, we continue with the same type of approach, but address the difficulty of the Darcy problem
in primal form. We remark that the problem to be studied further is symmetric and includes
an explicit variable, the Lagrange multiplier, on $\Gamma$. In this respect it differs
from the more common primal formulation, which leads to a non-symmetric system to be
solved for $\uf$, $p_f$ and $p_p$. Well-posedness of the latter problem was established in
\cite{discacciati2009navier} with efficient solvers proposed and analyzed e.g. in 
\cite{discacciati2003analysis, cai2009preconditioning, discacciati2007robin}.

An outline of the paper is as follows: Section \ref{sec:prelims} describes the notation,
introduces the symmetric primal Darcy-Stokes problem and illustrates the difficulties
in its preconditioning. The main challenge for the solver construction, i.e. the proper
posing of the coupling operator, is addressed in Section \ref{sec:normalderivative}.
Parameter robust preconditioners are then established in Section \ref{sec:darcy-stokes}.

%In Section \ref{sec:darcy-stokes},
%the general framework for preconditioning (a class of) coupled multiphysics problems of \cite{paper1}
%is recalled, and applied to establish a robust preconditioner for the monolithic Darcy-Stokes problems with the Darcy problem
%in primal form,
%which to the authors' knowledge has not yet been accomplished. 
%As our results rely on an assumption of numerical stability for an auxillary problem, this assumption is then
%investigated numerically and heuristically justified in \Cref{ex:bab_uni}.

%%%%%%%%%%%%%%%%%%%%%%%%%%%%%%%%%%%%%%%%%%%%%%%%%%%%%%%%%%%%%%%%%%%%%%

\section{Preliminaries}\label{sec:prelims}

Let $\Omega$ be a bounded Lipschitz domain in $\reals^n$, $n$=2 or 3, and denote
its boundary by $\partial \Omega$. We denote by $L^2(\Omega)$ the Lebesgue space of
square integrable functions, with the norm $\|u\|^2_{L^2(\Omega)} = \int_{\Omega} |u|^2 \, dx$,
and by $H^1(\Omega)$ the Sobolev space of functions with first derivative in $L^2(\Omega)$
with norm $\|u\|^2_{H^1(\Omega)} = \|u\|^2_{L^2(\Omega)} + \|\nabla u\|^2_{L^2(\Omega)}$. Note that
the spaces are both Hilbert spaces, with the standard inner products. These spaces are
defined in the same way when $u$ is a vector field, in which case we will write $\mathbf{u}$
in boldface. We also define the subspace $H^1_0(\Omega)$ to be the completion in $\|\cdot\|_{H^1(\Omega)}$
of $C^{\infty}_0(\Omega)$, the space of smooth functions on $\Omega$ whose restriction to $\partial \Omega$ is zero.  

For a Lipschitz domain $\Omega$ with $\Gamma \subset \partial\Omega$, we can define a trace operator $T$ by
$Tu = u \rvert_{\Gamma}$ for smooth $u$. This can be extended to a bounded, surjective
and right-invertible operator $H^1(\Omega) \to H^{\half}(\Gamma)$ (cf. e.g. ~\cite{ding1996proof}),
where the space $H^{\half}(\Gamma)$ will be defined later. Given a subset $\partial \Omega_D$ of $\partial \Omega$,
we let $H^1_{0, \partial \Omega_D}(\Omega)$, or for readability just $H^1_{0, D}(\Omega)$, be the subspace
of $H^1(\Omega)$ for which the restriction to $\partial \Omega_D$ is zero, where the restriction is
defined in terms of the trace operator. Typically, $\partial \Omega_D$ will be the subset of
$\partial \Omega$ on which Dirichlet conditions are prescribed.
%Unless otherwise specified, we assume that $\partial \Omega_D$ has positive measure.
We also define the semi-norm $L^2_{\mathbf{\tau}}(\Gamma)$ on $H^1(\Omega)$ to be the $L^2(\Gamma)$ norm
of the tangential component of $\mathbf{u}$ at $\Gamma$. In 2D, this is just
$\|\mathbf{u}\rvert_{\Gamma} \cdot \boldsymbol{\tau}\|_{L^2(\Gamma)}$ where $\mathbf{\tau}$ is a
tangent unit vector, while in 3D it is more conveniently written as
$\|\mathbf{u}\rvert_{\Gamma} - (\mathbf{u}\rvert_{\Gamma} \cdot \mathbf{n}) \mathbf{n}\|_{L^2(\Gamma)}$.

For any inner product space $X$, we let $(\cdot, \cdot)_X$ denote its inner product.
When $X=L^2(\Omega)$, we will omit the subscript if there is no cause for confusion.
We write the space of continuous linear operators from $X$ to $Y$ as $\mathcal{L}(X, Y)$,
or just as $\mathcal{L}(X)$ if $Y=X$. For any two Sobolev spaces $X, Y$ both contained in
a common ambient space, we define the intersection and sum spaces $X \cap Y$ and $X + Y$ in terms of the norms

  \begin{align*}
  \|u\|^2_{X \cap Y} = \|u\|^2_{X} + \|u\|^2_{Y}\quad\text{ and }\quad\|u\|^2_{X + Y} = \inf_{\substack{x + y = u\\  x \in X, y \in Y}} \|x\|^2_{X} + \|y\|^2_{Y}.
  \end{align*}

  For any $c>0$, we define the scaled space $cX$ to be just $X$ as a set, but with the inner
  product $(u, v)_X = c(u, v)_{X}$. Its norm is trivially equivalent to
  $\| \cdot \|_X$, but because the equivalence constant depends on $c$, the
  distinction between the two norms becomes important when we need to establish
  the independence of bounds with respect to problem parameters. 
  
  We define the fractional space $H^s(\Gamma)$ following \cite{kuchta2016preconditioners}.
  Let $S \in \mathcal{L}(H^1(\Gamma))$ be the operator such that $(Su, v)_{H^1} = (S (I-\Delta) u,v)  = (u, v)_{L^2}$
  for all $v \in H^1(\Gamma)$.  We can then find a basis of $H^1(\Gamma)$ of orthonormal eigenfunctions
  $e_i$ of $S$ with eigenvalues $\lambda_i>0$. Writing $u = \sum_i c_i e_i$ in this basis,
  we define the norm $\|u\|^2_{H^s(\Gamma)} = \sum c_i^2 \lambda_i^{-s}$ for any $s \in [-1, 1]$.
  Further, let the space $H^s(\Gamma)$ be the completion of $C^\infty(\Gamma)$ with respect to $\|\cdot\|_{H^s(\Gamma)}$.
  We also define the space $H_{00}^s(\Gamma)$ in the same manner, except that we then apply Dirichlet boundary conditions by choosing $S$ in
  $\mathcal{L} \left ( H^1_0(\Gamma)\right )$.
  Furthermore,  $H_{00}^s(\Gamma)$ is the completion of  $C^\infty_0(\Gamma)$ rather than $C^\infty(\Gamma)$. 

% this section due to Miro and left untouched
For the sake of completeness we review here the construction of a matrix realization
of fractional operators given in \cite{kuchta2016preconditioners}. To this end let
$V_h\subset H^1(\Gamma)$, $n=\dim V_h$ be a finite dimensional finite element subspace with basis
functions $\phi_i$, $i=1, \dots, n$ and $\mat{A}$, $\mat{M}\in\reals^{n\times n}$ be
the symmetric positive definite (stiffness and mass) matrices such that
\[
\mat{A}_{ij}=(\nabla \phi_j, \nabla \phi_i)\quad\text{ and }\quad \mat{M}_{ij}=(\phi_j, \phi_i).
\]
In case $V_h \not \subset H^1(\Gamma)$ and piecewise  constant (P0) discretization is
used we let
\[
\mat{A}_{ij}= \displaystyle\sum_{\nu \in \mathcal{N}} \avg{h}_\nu^{-1}(\jump{\phi_j}_\nu, \jump{\phi_i}_\nu)_{\nu},
\]
where $\mathcal{N}$ is a set of all the facets of the finite element
mesh. Further the (facet) average and jump operators are defined as
$\avg{u}_\nu=\tfrac{1}{2}(u|_{K^{+}}+u|_{K^{-}})$, $\jump{u}_\nu=u|_{K^{+}}-u|_{K^{-}}$
with $K^{+}$ and $K^{-}$ the two cells sharing facet $\nu$. When $\nu$ is an exterior facet,
  we define $\jump{u}_{\nu} = \avg{u}_{\nu} =  u \rvert _{K}$, where $K$ is the unique cell with $\nu$ as facet.

It follows that the generalized eigenvalue problem $(\mat{A}+\mat{M})\mat{U}=\mat{M}\mat{U}\mat{\Lambda}$
has only positive eigenvalues and a complete set of eigenvectors that form the basis of $\reals^n$ so that
the powers of $\mat{S}=\mat{U}\mat{\Lambda}(\mat{M}\mat{U})^{T}$ are well defined. For $s\in\left[-1, 1\right]$
we then set $\mat{H}(s)=\mat{M}\mat{S}^s$. Letting $\mat{u}$ be the vector of degrees of
freedom of $u_h\in V_h$, i.e. $u_h=\sum^n_i(\mat{u})_i\phi_i$, we finally have
\[
\norm{u_h}_{H^s}=\sqrt{\sum_{i, j=1}^{n} \mat{u}_i \left(\mat{H}_{ij}(s)\mat{u}_j\right)}.
\]

% trace properties

When $\mathbf{u}$ is a vector function, we define the normal trace
$T_\nn\mathbf{u} = \mathbf{u}|_{\Gamma} \cdot \nn $ using the trace operator $T$ component-wise.
As such $T_{\nn}$ is a continuous map $H^1(\Omega) \to H^{\frac 12}(\Gamma)$. Moreover, we 
let $\tG$ be the tangential trace operator. We remark that in 2D and 3D the operator maps to scalar, respectively vector fields. The normal derivative, $\Tgn u = \nabla u \cdot \nn |_{\Gamma}$, is more challenging
to define properly in this context. Let us therefore briefly sketch an approach,
which at least in the authors' opinion at first glance seems like a natural starting point.
However, as we will show, the approach does not yield robust preconditioners in our context.  
First, notice that if we impose additional regularity on $u$ and require that
$\Delta u \in L^2$ then $\Tgn$ is well defined. In detail, let $w \in H^{1/2}(\partial\Omega)$ and 
$E:H^{1/2}(\partial\Omega)\rightarrow H^1(\Omega)$ be a (harmonic) extension operator.
Then  $\Tgn u$ clearly lies in $H^{-1/2}(\partial \Omega)$ because  
\[
\int_{\partial \Omega} \Tgn u \cdot \ww \, ds = \int_{\Omega} \Delta u \cdot E \ww \, \dx  + \int_{\Omega} \nabla \cdot (E\ww) \cdot \nabla u \, \dx \le \infty . 
\]
This extra regularity assumption is, however, hard to express in the operator preconditioning framework.
In particular, to the author's knowledge, there are no \emph{standard} finite elements that
would enable us to exploit the extra regularity. 
A possible approach could be NURBS~\cite{hughes2005isogeometric} or 
$C^1$ discretizations developed for fourth order problems. However, the latter often show poor performance for second order problems~\cite{nilssen2001robust}.
%An alternative approach could be NURBS~\cite{hughes2005isogeometric}.  
%\mk{What about $C^1$ elements? Don't they have this property?}.
%\kam{I am actually not sure how NURBS etc fit into this picture. We can ask Jarle. Should be a good problem for him.}

Alternatively, we may attempt to define $\Tgn$ as a composition of the first order derivative operator, $\nabla$,  with
the 1/2 order normal trace operator $T_{\nn}$. The composition $\Tgn$ could then be expected to be a 3/2 operator
$\Tgn: H^1(\Omega) \rightarrow H^{-1/2}(\partial \Omega)$. From an operator preconditioning point of
view, this would be feasible to realize, as we will see below.
However, as we will demonstrate, robustness will not be obtained if we realize $\Tgn$ as a 3/2 operator. In fact, 
robustness is only obtained if $\Tgn$ is a first order operator, $\Tgn: H^1(\Omega) \rightarrow L^2(\partial \Omega)$. 
We remark here that while the operator in a continuous setting is $\Tgn: H^1(\Omega) \rightarrow L^2(\partial \Omega)$, 
in the discrete setting we will include a scaling parameter, i.e. the mesh size, because we use the
finite element method. To see that this is reasonable, notice that for finite elements,
the mass matrix, as representation of the identity, is differently scaled in different dimensions. 
In Example \ref{ex:L2_trace} we detail the scaling in a simplified example.

In order to demonstrate why posing the $\Tgn$ operator properly is required, let us now formulate the coupled Darcy-Stokes problem, where the Darcy problem is in primal form.  
As a starting point, let the Lagrangian of the coupled problem be,   
\begin{align*}
L(\uf, p_f, p_p, \lambda) &= 
\int_{\Omega_f} \frac{1}{2} \left(\mu (\nabla \uf)^2 - \ff\cdot\uf\right) \, \dx + \int_{\Gamma} \frac{1}{2} D (\uf\cdot \tau)^2 \, \ds        
+ \int_{\Omega_p} \frac{1}{2} K \left((\nabla p_p)^2 - g \, p_p \right) \, \dx \\    
&+ \int_{\Omega_f} \nabla\cdot \uf \, p_f \, \dx    
+ \int_{\Gamma} (T_n \uf - K \Tgn p_p )  \lambda \, \ds    
\end{align*}
Note that the sign of $p_f$ has been changed from \eqref{eq:stokes1}. Here,  the Lagrange multiplier $\lambda$ in $\int_{\Gamma} (T_n \uf - K \Tgn p_p )  \lambda \, \ds$ is
used to ensure mass conservation, while the extra term $\int_{\Gamma} \frac{1}{2} D (\uf\cdot \tau)^2 \, \ds$, where $D=\alpha_{\text{BJS}}\sqrt{\tfrac{\mu}{K}}$, corresponds to the Beavers-Joseph-Saffman condition \cite{mikelic2000interface}.

The corresponding weak formulation is obtained by the first order optimality conditions of the 
Lagrangian, that is; 
$\frac{\partial L}{\partial \uf} = 0$,   $\frac{\partial L}{\partial p_f} = 0$,   
$\frac{\partial L}{\partial p_p} = 0$, and   $\frac{\partial L}{\partial \lambda} = 0$.    
A variational formulation hence reads:
% Find $\uf, \up, p_p, p_f, w$ in \\ $\smuH \times \ismuL \times \isKHd \times \sKL \times X_?$
%Letting $V_f, Q_f, Q_p, X_? := \smuH , \ismuL ,  \isKHL, X_?$, 
Find $(\uf,  p_p, p_f, \lambda)$ such that  % in $V_f \times Q_f \times Q_p \times X_?$ \\ 
% TODO: replace V_f, V_p
\begin{equation}
  \label{eq:darcy_stokes_weak}
\begin{aligned}
a((\uf, p_p), (\vf, q_p)) + b((\vf, q_p), (p_f, \lambda)) &= f((\vf, q_p)) &\forall (\vf, q_p),\\ 
  b((\uf, p_p), (q_f, w)) &= g((q_f, w)) &\forall (q_f, w),
\end{aligned}
\end{equation}
where the bilinear forms $a$, $b$ are defined as
\begin{equation}
\label{eq:dsnaive}
\begin{aligned} 
  a((\uf, p_p), (\vf, q_p)) &=  \mu  (\nabla \uf, \nabla \vf)_{\Omega_f} + D (\uf \cdot \tau, \vf \cdot \tau)_\Gamma  + K (\nabla p_p, \nabla q_p)_{\Omega_p}, \\  
b((\uf, p_p), (q_f, w)) &=  (\nabla \cdot \uf, q_f)_{\Omega_f} + (T_{n}\uf, w)_{\Gamma} - K (\Tgn p_p, w)_{\Gamma}. 
\end{aligned}
\end{equation}
We shall refer to \eqref{eq:darcy_stokes_weak} as the (primal) Darcy-Stokes problem. Note
that the resulting formulation is symmetric.

While appropriate function spaces are readily available for $\uf,  p_p, p_f$ and their corresponding test functions, it is
less clear what the appropriate requirements are for $w$ and $\lambda$. This will be addressed below.    

%\begin{example}{Assuming that $-\Delta u \in L^2 $ or $\Tgn: H^1\rightarrow H^{-1/2}$}
\begin{example}{\textbf{Preconditioner for coupled Darcy-Stokes problem assuming $\Tgn: H^1\rightarrow H^{-1/2}$}.} \label{ex:prelim}
  Let us assume that $\Tgn$ is a 3/2 operator so that $K \Tgn p_p \in \frac{1}{\sqrt{K}} H^{-1/2}$ for $p_p\in \sKH$.
  Next, observe that since $\uf\in \smuH$ then $T_n \uf \in \sqrt{\mu} H^{1/2}$.
  Per assumption the coupling term
  $T_{n}\uf -K \Tgn p_p $ is $\in \sqrt{\mu} H^{1/2} + \frac{1}{\sqrt{K}} H^{-1/2}$ so that
  the dual variable $w \in  \frac{1}{\sqrt{\mu}} H^{-1/2} \cap \sqrt{K}H^{1/2}$.  
  In turn, we consider the following weak formulation:
	Find $\uf,  p_p, p_f, \lambda \in \smuH , \ismuL ,  \sKH, \frac{1}{\sqrt{\mu}} H^{-1/2}\cap \sqrt{K} H^{1/2}$ such that  
%$V_f, V_p, Q_f, Q_p, X_?$(definitions given in \Cref{tab:DSfunctionspaces})
% TODO: replace V_f, V_p
\begin{equation}
  \label{eq:example_naive}
\begin{aligned}
a((\uf, p_p), (\vf, q_p)) + b((\vf, q_p), (p_f, \lambda)) &= f((\vf, q_p) &\forall (v_f, q_p)\in \smuH\times\ismuL,\\ 
  b((\uf, p_p), (q_f, w)) &= g((q_f, w)&\forall (q_f, w)\in \sKH\times \frac{1}{\sqrt{\mu}} H^{-1/2}\cap \sqrt{K} H^{1/2}.
\end{aligned}
\end{equation}
The coefficient matrix associated with \eqref{eq:darcy_stokes_weak} reads
  \begin{align}
\label{coeff:darcy:stokes}
\AA =    \left( \begin{array}{cc|cc}
     -\mu \Delta + D \tG' \tG &   &(\nabla \cdot)' & T_{\nn}' \\ 
      &- K \Delta  & & - K \Tgn' \\ \hline
     \nabla \cdot&   & &  \\ 
     T_{\nn} & - K \Tgn  &  & 
     \end{array} \right).
  \end{align}

  Assuming that the proposed spaces indeed lead to well-posed operator $\mathcal{A}$,
  the operator preconditioning framework \cite{mardal2011preconditioning} yields as
  a preconditioner the Riesz mapping
\begin{equation}
\BB =    \left( \begin{array}{cccc} 
     -\mu \Delta + D \tG^{\prime}\tG &     \\ 
 &-K \Delta   & &   \\ \hline
                                                      & &  \frac{1}{\mu}I  &  \\
       & &    &  \frac{1}{\mu}\left ( I + \Delta \right )^{-1/2}+K \left ( I + \Delta \right )^{1/2}    \\ 
     \end{array} \right)^{-1}.\label{eq:B_darcy_stokes_naive}
 \end{equation}

In order to test the preconditioner, we solve problem
\eqref{eq:example_naive} on $\Omega=[0, 2] \times [0, 1]$,
 where $\Omega_f=[0, 1] \times [0, 1]$ and $\Omega_p=[1, 2] \times [0, 1]$
 and the Dirichlet boundary domains are $\partial\Omega_{f, D}=\set{(x, y)\in\partial\Omega_f, x = 0 }$
 and  $\partial\Omega_{p, D}=\set{(x, y)\in\partial\Omega_p, x = 2}$ , cf. \Cref{fig:DSdomain}.
 The mesh is a uniform triangular mesh, consisting of $4N^2$ equally sized isosceles triangles.
 To discretize \eqref{eq:darcy_stokes_weak}, we use lowest order (P2-P1) Taylor-Hood elements for
 the Stokes velocity and pressure, while piecewise quadratic elements (P2) were used for the
 Darcy pressure and piecewise constant elements (P0) for the Lagrange multiplier. Discretization is carried out in the FEniCS library \cite{fenics},  with coupling maps between the interface and domains and the fractional Laplacians being implemented by the extension FEniCS\textsubscript{ii} \cite{fenics_ii}. 
 %\karl{TODO: This is also done in Example 3.2. Should we state this somewhere in the preliminaries instead of citing them in each numerical example?} \kam{We can refer back to this example in the later example. }

 Approximation of the preconditioner  \eqref{eq:B_darcy_stokes_naive} is constructed by using
 single sweep of $V$-cycle of algebraic multigrid BoomerAMG from the Hypre library \cite{hypre} for
 all the blocks except for the interface block, which is inverted exactly. Starting
 from a random initial vector, we count the number of iterations required
 to solve the preconditioned linear system using the MINRES solver from the PETSc library \cite{petsc} with convergence
 criterion based on relative tolerance of $10^{-8}$ and absolute tolerance of $10^{-10}$. Additionally,
 the condition numbers of $\BB^{-1} \AA$ are computed using an iterative solver from the SLEPc library \cite{Hernandez:2005:SSF}.
 In the condition number computations the operator $\mathcal{B}$ is computed exactly, that is, all
 the blocks are inverted by LU. We remark that the solver setup should be used also in the 
 subsequent examples.

 The results of the experiment are plotted in Figure \ref{tab:naiveitercondtable}. By the failure
 of the iteration counts to stabilize, we see that using
 $\frac{1}{\mu}\left ( I + \Delta \right )^{-1/2}+K \left ( I + \Delta \right )^{1/2}$ as multiplier
 space does not lead to a robust preconditioner over the whole parameter range. Note, however, that in the regime where $\mu$ is significantly smaller than $K$ (i.e. the lower left region of the plots in Figure \ref{tab:naiveitercondtable}), iteration counts and condition numbers appear to be stable as the mesh is refined. In this regime, the norm of the multiplier space is dominated by the part from $\frac{1}{\sqrt{\mu}} H^{-1/2}$, 
 which is determined by posing of the trace operator. This suggests that the choice of $\sqrt{K} H^{1/2}$, i.e. wrong posing of the $\Tgn$ operator, is responsible for the lack of boundedness.

 % \begin{center}
%    { \small
%      \begin{table}
%        \centering

%        \begin{minipage}{0.45\linewidth}
%          \input{prelimtable.tex}
%        \end{minipage}
%        \begin{minipage}{0.45\linewidth}
%          \input{prelimtable2.tex}
% \end{minipage}

%        \caption{        
% Table of iterations and condition numbers for Example \ref{ex:prelim}. % No convergence
%          % within 500 iterations is indicated by $-$.
%          In all cases, $\alpha_{BJS}=1$.        
%          \label{tab:naiveitercondtable}
%        }
%      \end{table}
%    }
%  \end{center}
\begin{figure}
  \begin{center}
   { \small
       \centering
       \includegraphics[width=0.49\textwidth]{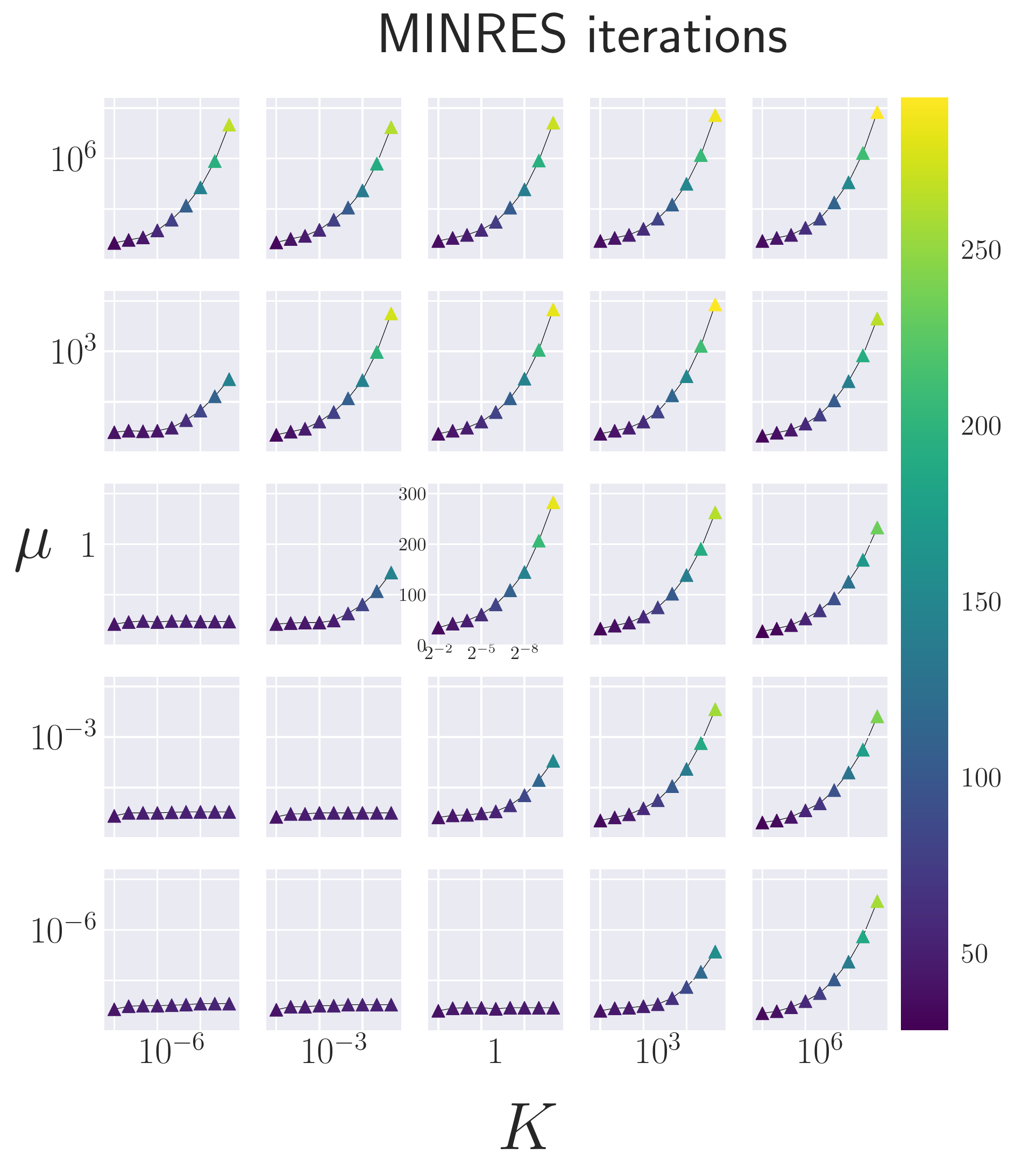}
       \includegraphics[width=0.49\textwidth]{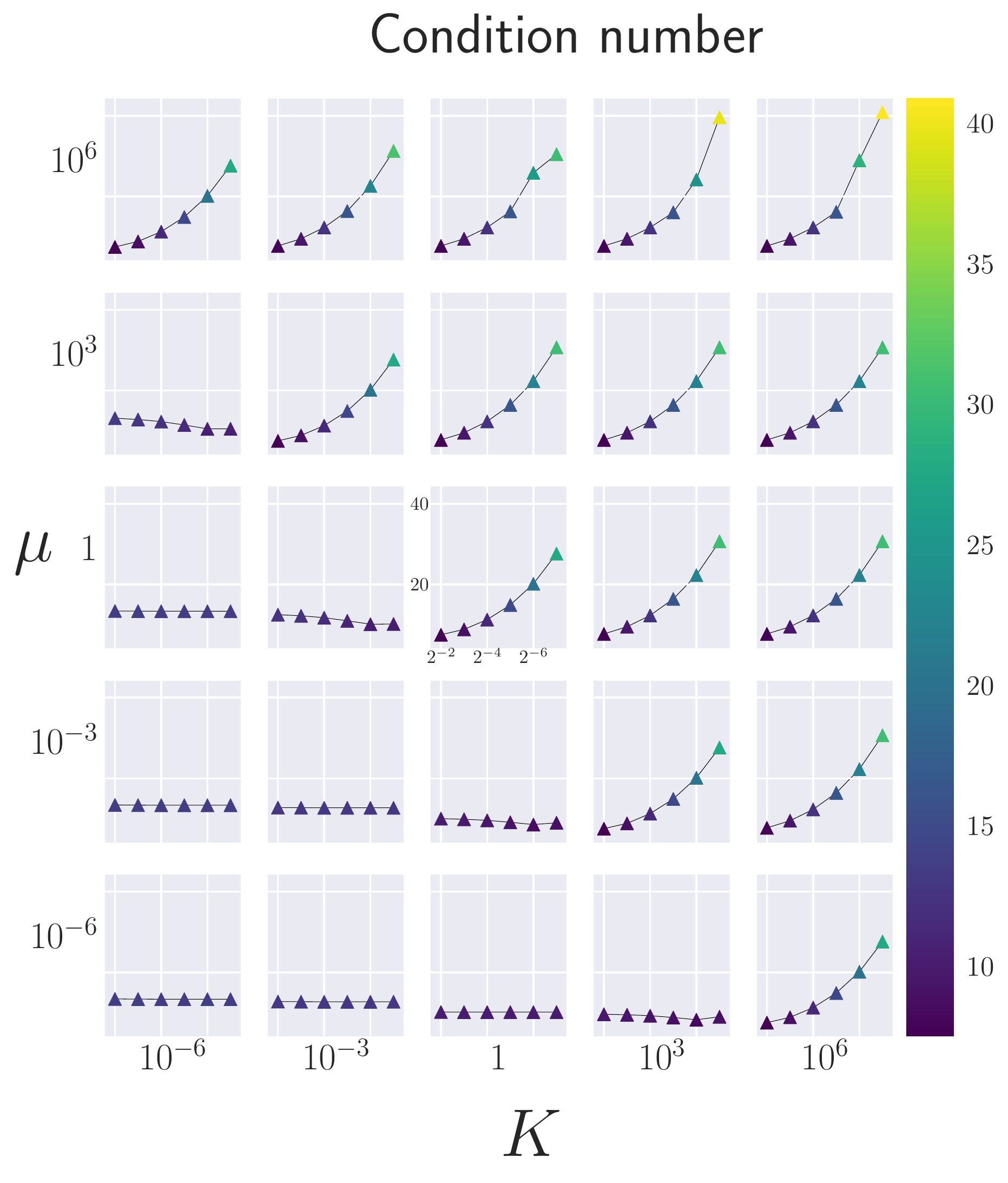}
  
       \caption{
         Mesh refinement vs. iteration counts (left) and condition numbers (right) for Example \ref{ex:prelim}.
         All subplots share $x$- and $y$-axes. 
         For fixed $\mu$, $K$ the 
         x-axis range in the iterations subplot extends from $h=2^{-2}$ to $h=2^{-10}$. In the condition number plots the range is 
         from $h=2^{-2}$ to $h=2^{-8}$.  
         In all cases, $\alpha_{BJS}=1$. %Triangles on top of each other look like squares.
         \label{tab:naiveitercondtable}}
   }
     \end{center}
   \end{figure}
\end{example}
%
%%%%%%%%%%%%%%%%%%%%%%%%%%%%%%%%%%%%%%%%%%%%%%%%%%%%%%%%%%%%%%%%%%%%%%%

\section{Approximating the trace normal gradient operator}
\label{sec:normalderivative}
A crucial step in the analysis of the Darcy-Stokes problem will be the mapping properties
of the operator $\Tgn$. As a computationally practical choice of space for the Darcy pressure 
is $\sqrt{K}H^1$, we immediately run into the problem discussed in the preliminaries because $\Tgn$
cannot be defined on all of $H^1$. This necessitates either an assumption of extra regularity or
an alternative approach.

Motivated by the observation in~\cite{holter2017sub}, that in a discrete finite element
setting the trace operator is stable as a map $L^2(\Omega) \rightarrow L^2(\partial \Omega)$, we
propose an alternative approach to construct the preconditioners. We start off by outlining the construction of an 
operator $\Tgne: H^1(\Omega_p) \to L^2(\Gamma)$ which will be an approximation to $\Tgn$.  Suppose $\Gamma$ is a
sufficiently regular subset of $\partial \Omega_p$, 
and that $\Gamma$ is of co-dimension  1 in $\Omega_p$.
The $\epsilon$-thick envelope $\Gamma_\epsilon = \{y \in\Omega_p, \operatorname{dist}(y, \Gamma) < \epsilon \}$
is a higher-dimensional approximation of $\Gamma$. For any  $v \in H^1(\Omega_p)$,  
\begin{equation}
  \label{eq:epsilontozero}
  \frac 1 {\epsilon} \int_{\Gamma_{\epsilon}}v \, \phi  \, \dx \to \int_{\Gamma} T v \, T \phi   \, ds \text{ as } \epsilon \to 0, 
\end{equation}
where $\phi$ is a test function in $H^1(\Omega_p)$.    

Note that although the integral over $\Gamma$ is not well-defined for a general
$v \in L^2(\Omega_p)$, the integral over $\Gamma_{\epsilon}$ is. Provided $\Gamma$ is
sufficiently regular and $\epsilon$ sufficiently small, we assume that there exists a vector field
$\nn_{\Gamma_{\epsilon}}$ on $\Gamma_{\epsilon}$ which approximates the normal vector $\nn_{\Gamma}$
of $\Gamma$ at $\Gamma$. Using $\nn_{\Gamma_{\epsilon}}$, we further assume that we can define a bounded extension
$E_{\epsilon}: L^2(\Gamma) \to L^2(\Gamma_{\epsilon})$ along  $\nn_{\Gamma_{\epsilon}}$
for which $\int_{\Gamma} w \, ds \approx \frac 1 {\epsilon} \int_{\Gamma_{\epsilon}} E_{\epsilon} w \, \dx$ for any $w \in L^2(\Gamma)$.
Provided $\nn_{\Gamma_{\epsilon}}$ and $E_{\epsilon}$ can be defined, then for any $u \in H^1(\Omega_p)$ we can define $\Tgne u$ by
\[
\int_{\Gamma} \Tgne u \cdot w \, ds = \frac 1 {\epsilon} \int_{\Gamma_{\epsilon}} \nabla u \cdot \nn_{\Gamma_{\epsilon}} E_{\epsilon}w \, dx
\]
for any $ w \in L^2(\Gamma)$, thus defining the required map $\Tgne: H^1(\Omega_p) \to L^2(\Gamma_\epsilon)$ approximating \Tgn. We assume that the resulting operator $\Tgne$ 
is both
surjective and bounded, with $\|\Tgne u\|_{L^2(\Gamma)} \leq C\|u\|_{H^1(\Omega_p)}$, and that $\Tgne$ has a bounded right inverse.

We emphasize that $\Tgne$ is just an analytical tool constructed for the analysis in
the continuous setting and that $\epsilon$ is not related to the mesh size $h$. In
fact, we can choose $\epsilon$ far smaller than the mesh size and for any practical
purposes in computations we assume that $\Tgne$ will be practically identical to $\Tgn$. We summarize the assumption as follows:

\begin{assumption}
\label{normalderivativeeps}
Given a sufficiently regular $\Gamma$,  $\Tgne : H^1(\Omega_p) \rightarrow L^2(\Gamma)$ is a bounded surjection
which approximates $\Tgn$ on the subspace of $H^1$ on which $\Tgn$ can be defined. Further, $\Tgne$ has a bounded right inverse.
\end{assumption}

Although characterizing the conditions under which \Cref{normalderivativeeps} holds is beyond the scope of this paper, we motivate the existence of the required constructions $E_{\epsilon}, \mathbf{n}_{\epsilon}$  in a few simple examples below. 
\begin{example}\label{ex:epsilonconstructions}
Let $\Gamma$ be the $y-$axis, and $\Omega_p$ be the positive half-plane. The construction of $E_{\epsilon}, \nn_{\Gamma_{\epsilon}}$ is then given by $\nn_{\Gamma_{\epsilon}} = \nn_{\Gamma} = (-1, 0)$ and for $w(y) \in C^1(\Gamma)$ we let $(E_{\epsilon} w)(x, y) = w(y)$. This continuously extends to all of $L^2(\Gamma)$.
Clearly $\Tgne u  \rightarrow \Tgn u$ as $\epsilon\rightarrow 0$ for $u \in C^1$. Given any $w(y) \in C^0(\Gamma)$, define $u$ by $u(x, y) = -xw(y)$. Then the map $w \to u$ continuously extends to a right inverse of $\Tgne$, as by linearity $\Tgne u = \Tgn u = w$.

Next, suppose $\Omega_p$ is the unit disk, and $\Gamma$ its boundary. By parametrizing $\Gamma$ with e.g. polar coordinates, this case can be effectively translated to the above. $\nn_{\Gamma_{\epsilon}}$ is now the unit radial vector $\mathbf{i}_r$, and for any $w(\theta) \in C^1(\Gamma)$,  $(E_{\epsilon} w)(r, \theta) = w(\theta)$. Again, this definition of $E_{\epsilon}$ extends to all of $L^2(\Gamma)$.
 Because $\frac 1 {\epsilon} \int_{\Gamma_{\epsilon}} \nabla u \cdot \nn_{\Gamma_{\epsilon}} E_{\epsilon}w \, dx = \int \limits_0^{2 \pi} w(\theta) \cdot  \int  \limits_{1-\epsilon}^1 \frac 1 {\epsilon} \frac {\partial u} {\partial r}  \, r dr \, d\theta $ and $\frac 1 {\epsilon} \int\limits^1_{1-\epsilon} f(r) \, dr \to f(1)$ as $\epsilon \to 0$,  we again have $\Tgne u  \rightarrow \Tgn u$ as $\epsilon\rightarrow 0$ for $u \in C^1$. Analogously to the previous case, a right inverse can be defined by sending any $w(\theta) \in C^0(\Gamma)$ to $u(r, \theta) = r w(\theta)$.

\end{example}

Before considering the Darcy-Stokes problem, we justify Assumption \ref{normalderivativeeps}. 
First we consider a simplified example in order to illustrate how the scaling of mass matrices in 
different dimensions affect preconditioners constructed via the application of trace operators. 
Then, in Example
\ref{ex:bab_uni} we construct preconditioners for a 
Poisson problem with a $\Tgn$-constraint which is to be enforced by a
Lagrange multiplier, cf. the Babu{\v s}ka problem \cite{babuvska1973finite} involving
the trace operator.

\begin{example}{\textbf{Trace constrained $L^2$ projection.}}\label{ex:L2_trace}
Let $\Omega$ be a bounded domain with $\Gamma\subseteq\partial\Omega$ and 
$V=H^1(\Omega)$. We then consider the problem 
\begin{equation}\label{eq:L2_min}
  \min_{u\in V}\int_{\Omega} {u}^2\,\dx - 2\int_{\Omega} f u \,\dx
  \quad\text{ subject to }\quad\int_{\Gamma}(T u - g) \, p \,\ds=0.
\end{equation}
Letting $p$ denote the Lagrange multiplier associated with the boundary
constraint, the extrema $u\in V$, $p\in Q=L^2(\Gamma)$ of
the Lagrangian of \eqref{eq:L2_min} satisfy the variational problem: 
Find $u\in V$ and $p\in Q$ such that 
\begin{equation}\label{eq:L2_N}
  \begin{aligned}
    \int_{\Omega} u v \,\dx + \int_{\Gamma}p T v \,\ds &= \int_{\Omega} f v \,\dx &\forall v\in V,\\
    \int_{\Gamma}q T u\,\ds &= \int_{\Gamma}g q\,\ds &\forall q\in Q.
  \end{aligned}
\end{equation}
The operator of the preconditioned continuous problem then reads
\begin{equation}\label{eq:precond_L2}
\mathcal{B}\mathcal{A} =
\begin{pmatrix}
  I & \\
  & S
\end{pmatrix}^{-1}
\begin{pmatrix}
  I & T^{\prime}\\
  T & 
\end{pmatrix},
\end{equation}
where $S$ is to be constructed such that the condition number is bounded in
the discretization parameter $h$. Here we shall consider three constructions. 
We remark that when using the finite element method, the identity or the mass matrix 
has eigenvalues such that both the smallest and the largest eigenvalues scale as $h^d$ on uniform mesh.
First we consider $S=I$, with eigenvalues $\approx h$. 
Then, following \cite{holter2017sub}, we let $S=h^{-1}I$, i.e., a matrix with eigenvalues $\approx 1$.
Finally, the choice of $S=(-\Delta+I)^{-1/2}$ is included to show that the
relevant trace space in \eqref{eq:precond_L2} is not (by viewing the trace as an order 1/2 operator)
$H^{1/2}$ so that dual variable would reside in $H^{-1/2}$.

We remark that the first two operators are in practical computations assembled as
weighted mass matrices where the weights for the respective operators
are 1 and inverse cell volume. Recalling Preliminaries \S\ref{sec:prelims}
the matrix representation of the fractional operator is $\mat{H}(1/2)$.

To compare the three preconditioners, we let $\Omega$ be a unit square,
$\Gamma=\left\{(x, y)\in\partial\Omega,x=0\right\}$. Further, the domain shall be 
discretized uniformly into $4N^2$ isosceles triangles with size $h=1/N$, see Figure \ref{fig:bab_uni}.
Considering finite element discretization by P2-P1 elements Table \ref{tab:L2_uni}
lists spectral condition numbers of \eqref{eq:precond_L2}. It can be seen
that only the $S=h^{-1}I$ preconditioner leads to results independent of $h$.

% How to compute
%We remark that the condition numbers of \eqref{eq:precond_L2} were computed
%from the full spectrum by LAPACK \cite{laug} if the system has fewer than
%10 thousand unknowns. For larger problems (e.g. for $h=2^{-8}$ the system size
%exceeds 500 thousand) we use an iterative solver from SLEPc \cite{Hernandez:2005:SSF}
%with convergence tolerance $10^{-3}$. Here $\mathcal{B}$ is computed exactly.

The growth of the condition number in Table \ref{tab:L2_uni} due to the 
preconditioner with $-1/2$ power indeed confirms that $H^{1/2}$ is not
appropriate in our setting. An attempt to establish the trace space could
be based on viewing the trace as an 1/2 operator. Starting from $L^2$ a formal
calculation then leads to the space $H^{-1/2}$ and $H^{1/2}$ as the multipler
space. While we do not include here the results for $S=(-\Delta+I)^{1/2}$ we remark
that the condition number behaves practically identically to $S=I$.

\begin{minipage}{\textwidth}
\begin{minipage}{0.45\textwidth}
  \begin{center}
  \footnotesize{
      \begin{tabular}{c|ccc}
        \hline
$h$  & $I$ & $(-\Delta+I)^{-1/2}$ & $h^{-1} I$\\
    \hline
$2^{-2}$ & 8.72  &24.08   & 4.63\\
$2^{-3}$ & 12.11 &47.84   & 4.63\\
$2^{-4}$ & 16.91 &95.15   & 4.63\\
$2^{-5}$ & 23.70 &189.9   & 4.63\\
$2^{-6}$ & 33.31 &379.2   & 4.63\\
$2^{-7}$ & 46.90 &758.0   & 4.63\\
$2^{-8}$ & 66.12 &1515    & 4.63\\
  \hline
      \end{tabular}
      \captionof{table}{Condition numbers of \eqref{eq:precond_L2} with different preconditioners and
        discretization by P2-P1 elements on (us) mesh from Figure \ref{fig:bab_uni}. Boundedness is obtained with the Schur complement preconditioner $h^{-1}I$.
        }
      \label{tab:L2_uni}
  }
  \end{center}
\end{minipage}
\begin{minipage}{0.45\textwidth}
  \begin{center}
    \footnotesize{
            \begin{tabular}{c|ccc|ccc} \hline
              \multirow{2}{*}{$l$} & \multicolumn{3}{c|}{P2-P1} & \multicolumn{3}{c}{P2-P0}\\
              \cline{2-7}
              & (us) & (uu)  & (nu) & (us) & (uu)  & (nu)\\
              \hline
1& 4.63 & 4.63 & 4.10 & 4.63 & 4.63 & 3.98\\ 
2& 4.63 & 4.06 & 4.32 & 4.63 & 4.07 & 4.33\\ 
3& 4.63 & 4.20 & 4.28 & 4.63 & 4.20 & 4.31\\ 
4& 4.63 & 4.29 & 4.31 & 4.63 & 4.32 & 4.34\\ 
5& 4.63 & 4.45 & 4.50 & 4.63 & 4.43 & 4.45\\ 
6& 4.63 & 4.25 & 4.28 & 4.63 & 4.32 & 4.37\\ 
7& 4.63 & 4.25 & 4.36 & 4.63 & 4.28 & 4.39\\
\hline
            \end{tabular}
            }
  \end{center}
\captionof{table}{Condition numbers of \eqref{eq:precond_L2} with 
preconditioner using $S=h^{-1}I$. Boundedness with different types of 
triangulations, cf. Figure \ref{fig:bab_uni}, and discretizations can be observed.
}
\label{tab:L2}
\end{minipage}
\end{minipage}

In order to verify that the properties of $h^{-1}I$ preconditioner are not
due to the highly structured mesh, we consider two additional discretizations
of $\Omega$ shown in Figure \ref{fig:bab_uni}. In particular, the triangulations are
obtained as refinements of the unstructured meshes where in one case
the mesh size is uniform while in the other one the mesh is finer close to the
multiplier domain $\Gamma$. Moreover, using these triangulations, problem \eqref{eq:precond_L2}
shall be discretizated by P2-P1 elements as well P2-P0 elements to provide
more evidence for the preconditioner construction. Indeed, Table \ref{tab:L2}
shows that the condition numbers of \eqref{eq:precond_L2} are bounded
irrespective of the underlying mesh and the finite element discretization considered.
\end{example}

\begin{example}{\textbf{Babu{\v s}ka problem with Neumann boundary conditions.}}\label{ex:bab_uni}
Let $\Omega$ be a bounded domain with the boundary partitioned into non-overlapping subdomains
  $\partial\Omega=\partial\Omega_D\cup\partial\Omega_N\cup \Gamma$ such that $\semi{\partial\Omega_D}>0$
  and $\semi{\Gamma}>0$. We will consider both the case that  
$\partial\Omega_D\cap\Gamma=\emptyset$ and later the case that $\partial\Omega_N\cap\Gamma=\emptyset$.  
Let $V=H^1_{0, \partial\Omega_D}(\Omega)$ and consider the problem
\begin{equation}\label{eq:babuska_min}
  \min_{u\in V}\int_{\Omega} \semi{\nabla u}^2\,\dx - 2\int_{\Omega} f u \,\dx
  \quad\text{ subject to }\quad\int_{\Gamma}(\Tgn u - g) \, p \,\ds=0.
\end{equation}
With $p$ the Lagrange multiplier associated with $\Tgn$-constraint
\eqref{eq:babuska_min} leads to a variational problem: Find $u\in V$ and $p\in Q=L^2(\Gamma)$ such that 
\begin{equation}\label{eq:babuska_N}
  \begin{aligned}
    \int_{\Omega} \nabla u\cdot \nabla v \,\dx + \int_{\Gamma}p \Tgn v\,\ds &= \int_{\Omega} f v \,\dx &\forall v\in V,\\
    \int_{\Gamma}q \Tgn u\,\ds &= \int_{\Gamma}g q\,\ds &\forall q\in Q.
  \end{aligned}
\end{equation}
The preconditioned continuous problem then reads
\begin{equation}\label{eq:precond_bab}
\mathcal{B}\mathcal{A} =
\begin{pmatrix}
  -\Delta & \\
  & S
\end{pmatrix}^{-1}
\begin{pmatrix}
  -\Delta & {\Tgn}^{\prime}\\
  \Tgn & 
\end{pmatrix}.
\end{equation}
Following the preliminaries where $\Tgn$ was regarded as a 3/2 operator we let
$S=(-\Delta+I)^{1/2}$. Alternatively, $S=h^{-1}I$ is set following the Assumption
\ref{normalderivativeeps}. Finally $S=I$ is considered. Matrix realization
of the $S$ operators shall be identical to Example \ref{ex:L2_trace}. We shall also
use the tessellations described in Example \ref{ex:L2_trace} as well as identical
eigenvalue solvers.

%\kam{Here we have to described the discretization, in particular mass matrices, etc. }

To compare the three preconditioners we let $\Omega$ be a unit square and
$\Gamma=\left\{(x, y)\in\partial\Omega, x=0\right\}$ and we consider first the
(Neumann) case where $\partial\Omega_N=\left\{(x, y)\in\partial\Omega,y=0\text{ or }y=1\right\}$,
i.e. where the multiplier domain intersects the part of boundary with Neumann
boundary conditions. Using the uniform meshes (marked as (us) Figure in \ref{fig:bab_uni})
and P2-P1 elements, Table \ref{tab:bab_uni} shows the spectral condition numbers of
\eqref{eq:precond_bab}. As in Example \ref{ex:L2_trace} only $S=h^{-1}I$ preconditioner
(based on \Cref{normalderivativeeps}) leads to results independent of $h$.\\
\begin{minipage}{\textwidth}
\begin{minipage}{0.42\textwidth}
  \begin{center}
  \footnotesize{
      \begin{tabular}{c|ccc}
    \hline
$h$  & $(-\Delta+I)^{1/2}$ & $I$ & $h^{-1} I$\\
    \hline
$2^{-2}$ &    11.99&     6.70&       4.88\\
$2^{-3}$ &    14.55&     9.27&       4.88\\
$2^{-4}$ &    18.47&     12.89&      4.88\\
$2^{-5}$ &    24.44&     18.01&      4.88\\
$2^{-6}$ &    33.25&     25.26&      4.88\\
$2^{-7}$ &    45.96&     35.52&      4.88\\
$2^{-8}$ &    64.10&     50.02&      4.88\\
  \hline
      \end{tabular}
  }
  \end{center}
  \captionof{table}{Condition numbers of \eqref{eq:precond_bab} discretized
    by P2-P1 elements on uniform refinements of (us) mesh in Figure \ref{fig:bab_uni}.
    Boundednes in discretization is obtained only with $S=h^{-1}I$.
    }
  \label{tab:bab_uni}
\end{minipage}
\begin{minipage}{0.58\textwidth}
  \begin{center}
\includegraphics[width=0.32\textwidth]{img/0_80.mps}
\includegraphics[width=0.32\textwidth]{img/1_80.mps}
\includegraphics[width=0.32\textwidth]{img/2_80.mps}
\vspace{-10pt}
\captionof{figure}{Parent meshes for uniform refinement. From left to right:
  uniform structured(us), uniform unstructured(uu), non-uniform unstructured(nu).
  Non-uniform mesh has finer (by factor 3) mesh size close to $\Gamma$.
}
\label{fig:bab_uni}
  \end{center}
\end{minipage}
\end{minipage}

Table \ref{tab:bab_d_n} shows that the performance of $h^{-1}I$ in \eqref{eq:precond_bab}
remains robust if different tessellations and finite element discretizations are
used.\\
%\kam{Need to describe that $I$ is mass matrix}
% P2-P1 / Neumann and different meshes | Dirichlet
\begin{minipage}{\textwidth}
    \footnotesize{
      \begin{minipage}{0.49\textwidth}
          \begin{center}
            \begin{tabular}{c|ccc|ccc} \hline
              \multirow{2}{*}{$l$} & \multicolumn{3}{c|}{P2-P1} & \multicolumn{3}{c}{P2-P0}\\
              \cline{2-7}
              & (us) & (uu)  & (nu) & (us) & (uu)  & (nu)\\
              \hline
1& 4.88&4.77&6.64&3.49&3.49&3.06\\
2& 4.88&5.98&6.56&3.49&3.04&3.37\\
3& 4.88&5.78&5.67&3.49&3.24&3.36\\
4& 4.88&6.31&6.67&3.49&3.40&3.40\\
5& 4.88&5.25&5.68&3.49&3.44&3.48\\
6& 4.88&5.71&5.89&3.49&3.41&3.44\\
7& 4.88&6.14&6.61&3.49&3.35&3.47\\
\hline
\end{tabular}
\end{center}
      \end{minipage}
    }
          \footnotesize{
            \begin{minipage}{0.49\textwidth}
              \begin{center}
            \begin{tabular}{c|ccc|ccc} \hline
              \multirow{2}{*}{$l$} & \multicolumn{3}{c|}{P2-P1} & \multicolumn{3}{c}{P2-P0}\\
              \cline{2-7}
              & (us) & (uu)  & (nu) & (us) & (uu)  & (nu)\\
              \hline                  
1&5.34&5.25&6.67&3.48&3.45&3.04\\
2&5.34&6.25&6.67&3.49&2.99&3.37\\
3&5.34&5.94&5.84&3.49&3.24&3.36\\
4&5.34&6.52&6.93&3.49&3.40&3.40\\
5&5.34&5.56&6.07&3.49&3.44&3.48\\
6&5.34&5.91&6.17&3.49&3.41&3.44\\
7&5.34&6.40&6.85&3.49&3.35&3.47\\
\hline
            \end{tabular}
              \end{center}
              \end{minipage}
          }
          \captionof{table}{Condition numbers of \eqref{eq:precond_bab} using $S=h^{-1}I$ preconditioner discretized on uniform
            refinements of parent meshes in Figure \ref{fig:bab_uni} using two element types. Refinement
            level is indicated by $l$. (Left) $\Gamma$ intersects $\partial\Omega_N$.
            (Right) $\Gamma$ intersects $\partial\Omega_D$.
          }
          \label{tab:bab_d_n}
\end{minipage}
\vspace{0.2cm}

In the context of multiscale problems, compatibility of boundary conditions of the multiplier
space and the boundary conditions prescribed on the domain intesecting $\Gamma$
is known to present an issue, cf. e.g. \cite{galvis2007non}. Here, we address
this problem by considering \eqref{eq:precond_bab} with $\semi{\partial_N\Omega}=0$,
i.e. we let $\Gamma$ intersect only the Dirichlet boundary. We remark that until this
point only intersection with Neumann boundary was considered. 

In Table \ref{tab:bab_d_n} the Dirichlet problem is considered with an
\emph{unmodified} $h^{-1}I$ preconditioner. In particular, with P2-P1
discretization we impose \emph{no} boundary conditions on the multiplier space. 
Using this construction the condition numbers can be seen to remain bounded on
all the meshes and with both finite element discretizations.

We remark that the $h^{-1}I$ preconditioner is equally unaffected by the Dirichlet
boundary conditions on $\partial\Omega_D=\partial\Omega\setminus\Gamma$ in the
trace-constrained $L^2$ projection problem \eqref{eq:L2_min} with $V=H^1_{0,\partial\Omega_D}(\Omega)$,
cf. Example \ref{ex:L2_trace}, in contrast to the $H^1$ problems considered in \cite{paper1}, where
the appropriate preconditioner was $H_{00}^{-\half}$ or $H^{-\half}$ depending on whether the
interface intersected the Dirichlet boundary or not. We remark that 
in the continuous setting boundary values have measure zero and this may then be perceived as the $L^2$ space
being the correct one in our discrete setting. Of course, the counterargument in 
the continuous setting is that then 
the trace cannot be defined. However, in the discrete setting, this can be done.

Without including the simulation results we
comment here that the condition numbers of the Dirichlet problem are practically
identical to those presented in Tables \ref{tab:L2_uni} and \ref{tab:L2}. In addition, with the two preconditioners $S=I$ and $S=(-\Delta + I)^{1/2}$ on the
unstructured meshes a growth of condition numbers with $h$ is observed similar
to Table \ref{tab:bab_uni}.
\end{example}

We remark that the stability of the preconditioner $h^{-1}I$ in
Example \ref{ex:bab_uni} provides numerical evidence for well-posedness of
\eqref{eq:babuska_N}, i.e. the Darcy subproblem in the coupled Darcy-Stokes
system \eqref{eq:darcy_stokes_weak}.
%
% -end-babuska---------------------------------------------------------

\section{Robust Preconditioners for the Darcy--Stokes system}
\label{sec:darcy-stokes}

In Example \ref{ex:prelim}, we showed that the efficiency of the preconditioner \eqref{eq:B_darcy_stokes_naive} for the primal Darcy--Stokes problem  \eqref{coeff:darcy:stokes} varied substantially with 
the material parameters even
though the Stokes block and the Darcy block were preconditioned with  appropriate preconditioners, and argued that the reason was a poor preconditioner at the interface. 

In this section we demonstrate that robustness with respect to mesh resolution and variations in
material parameters can be obtained by posing the Lagrange multiplier in properly weighted fractional spaces, namely the intersection space $X_{\Gamma} = \DSmulte$. No modifications of the velocity or pressure space norms will be required.
Our analysis is closely related to \cite{paper1}, and based on \Cref{normalderivativeeps} along with 
an assumption of stability for the Stokes problem. We remark that although \Cref{normalderivativeeps} 
is motivated by the discrete problem, our analysis is carried out in a continuous setting. 

Let $\partial {\Omega_i} = \partial\Omega_{i,D} \cup \partial\Omega_{i,N} \cup \Gamma$ for $i=f, p$ such that $\partial\Omega_{f, D}\cap\Gamma=\emptyset$. We shall prove well-posedness of the coupled Darcy-Stokes problem \eqref{eq:darcy_stokes_weak} 
with spaces
\begin{equation}\label{eq:DS_spaces}
	V_f = \smuH,\, 
	Q_f = \ismuL,\,  
	Q_p = \sKH,\,  
	X_\Gamma = \DSmulte.  
\end{equation}
We remark that in case $\Gamma$ intersects only the Dirichlet boundary $\partial\Omega_{f, D}$ the 
space $H^{-1/2}$ needs to be modified to reflect $H^{1/2}_{00}$ as the appropriate trace space 
of $V_f$. We refer to \cite{paper1} for a thorough discussion of the subject.

As a prerequisite for the coupled problem to be well-posed, we require that each subproblem is well-posed. 
For the Stokes subproblem the property has been demonstrated by numerical experiments in \cite{paper1}. 
Here we state the result without proof.

\begin{assumption}\label{assumption:stokesbound}
Let $\Omega_f$ be such that $\partial {\Omega_f} = \partial\Omega_{f,D} \cup \partial\Omega_{f,N} \cup \Gamma$, $\semi{\partial\Omega_{f, D}}>0$ 
and $\partial\Omega_{f, D}\cap\Gamma=\emptyset$. We define $V_S = \smuH \times \ismuL \times \DSmultf$ 
and the forms
    \begin{align*} 
       a_S((\uf, p_f, \lambda), (\vf, q_f, w)) =& \mu (\nabla \uf, \nabla \vf) + D (\uf \cdot \mathbf{\tau}, \vf \cdot \mathbf{\tau})_{\Gamma} + (p_f, \nabla \cdot \vf)  + (\nabla \cdot \uf, q_f) + (T_n\uf, w)_{\Gamma} + (\lambda, T_n\vf)_{\Gamma},\\
      L_S((\vf, q_f, w)) =& (\mathbf{f}, \vf) + (g, q_f) + (h_D, w)_{\Gamma},
    \end{align*}
    where
    $\mathbf{f} \in {\frac 1 {\sqrt{\mu}}H^{-1}(\Omega_f)}, g \in {\sqrt{\mu}L^2(\Omega_f)}, h_D \in {\sqrt{\mu} H^{\half}(\Gamma)}$ are arbitrary. 
    Then we assume that the Stokes problem: Find $(\uf, q_f, \lambda)\in V_s$ such that
    \begin{align*}
      a_S((\uf, p_f, \lambda), (\vf, q_f, w)) = L_S((\vf, q_f, w))\quad\forall(\vf, q_f, w) \in V_S
    \end{align*}
    satisfies the Brezzi conditions and hence has a unique solution $(\uf, p_f, \lambda) \in V_S$ 
    and the following bound holds
    $$
      \| (\uf, p_f, \lambda)\|_{V_S} \leq C\left ( \|\mathbf{f}\|^2_{\frac 1 {\sqrt{\mu}}H^{-1}(\Omega_f)} + \|g\|^2_{\sqrt{\mu}L^2(\Omega_f)} + \|h_D\|^2_{\sqrt{\mu} H^{\half}(\Gamma)}  \right )^{\frac 1 2}.
    $$
    Here the constant $C$ depends only on $\Omega_f, \: \partial \Omega_{f, D}$ and $\Gamma$.
\end{assumption}

Corresponding well-posedness of the Darcy problem with $\Tgn$-constraint was demonstrated 
numerically for $K=1$ in Example \ref{ex:bab_uni}. Here, we analyze the general case.

\begin{lemma} \label{lmm:primaldarcybound}
  Suppose $\Omega_p, \Gamma$ are such that \Cref{normalderivativeeps} holds and $\semi{\partial\Omega_{p, D}}> 0$. Then for any $f \in \frac 1 {\sqrt K} H^{-1}(\Omega_p),$ $ h \in \frac 1 {\sqrt{K}} L^2(\Gamma)$, the problem of finding $(p_p, \lambda)\in \sKH \times \DSmultpe$ so that
\begin{equation}\label{eq:darcy_sys}
  \begin{aligned}
    K(\nabla p_p, \nabla q_p)_{\Omega_p} + K(\lambda, \Tgne q_p)_{\Gamma} &= (f, q_p)&\forall q_p\in \sKH,\\
    K(\Tgne p_p, w)_{\Gamma} &=(h, w)_{\Gamma}&\forall w\in \DSmultpe
  \end{aligned}
\end{equation}
  has a unique solution satisfying 
$$
  \|p_p\|_{\sKH} \leq C \left ( \|h\|^2_{\frac 1 {\sqrt{K}} L^2(\Gamma)} + \|f\|^2_{\frac 1 {\sqrt K} H^{-1}(\Omega_p)} \right )^{\frac 1 2},
$$
  where $C$ is a constant depending only on $\Omega_p$.
\end{lemma}

\begin{proof}
  Let $V=\sKH$, $Q=\DSmultpe$. We consider the left-hand side of \eqref{eq:darcy} 
  as an operator
  \begin{equation}\label{eq:brezzi}
  \begin{pmatrix}
   A & B'\\
   B   
  \end{pmatrix}:
  V\times Q\rightarrow V'\times Q',
  \end{equation}
  where
  $(A p_p, q_p) = K(\nabla p_p, \nabla q_p)_{\Omega_p}$ and $(B p_p, w) =  K(\Tgne p_p, w)_{\Gamma}$.
  
  The statement of the theorem follows from the Brezzi theory \cite{brezzi1974existence} once the Brezzi conditions 
  are verified. That is, we must show that $A$, $B$ are bounded, $A$ is coercive on $\ker B$ 
  and that the inf-sup condition $\inf_{q\in Q}\sup_{v\in V} (Bv, q) \geq \beta\norm{v}\norm{q}$ holds for some constant $\beta>0$.
  
  Here the boundedness of $A$ and the coercivity on $V$ are evident. For the latter we 
  recall that $\semi{\partial\Omega_{p, D}}> 0$ is assumed 
  and invoke the Poincare inequality. \Cref{normalderivativeeps} 
  is needed to show the properties of $B$. Because 
  $$K(\lambda, \Tgne q_p)_{\Gamma} \leq  K \|\lambda\|_{\DSmultpe} \|\Tgne q_p\|_{\frac 1 {\sqrt{K}} L_2(\Gamma)} \leq \|\Tgne\|  \|\lambda\|_{\DSmultpe} \|q_p\|_{\sqrt{K} H^1(\Omega_p)}, $$ we have boundedness with constant $\|\Tgne\|$. For the inf-sup condition, we recall the bounded right inverse $E$ of $\Tgne$. Letting $p_p^* = E(\lambda)$, we have $K(\lambda, \Tgne p_p^*) = \|\lambda\|^2_{\DSmultpe}$ and $\|p_p^*\|_{\sKH} \leq  \|E\| \| \lambda \|_{\sqrt{K}L_2(\Gamma)}$ so that
  \begin{align*}
    \sup\limits_{p_p \in \sKH} \frac {K(\lambda, \Tgne p_p)} {\|p_p||_{\sKH}} &\geq \frac {K(\lambda, \Tgne p_p^*)} {\|p_p^* ||_{\sKH}} = \frac { \|\lambda\|^2_{\DSmultpe}} {\|p_p^* ||_{\sKH}} \geq \frac { \|\lambda\|^2_{\DSmultpe}} {\|E\| \|\lambda ||_{\DSmultpe}}  \geq \frac 1 {\|E\|} \|\lambda\|_{\DSmultpe}.
  \end{align*}
This proves all the Brezzi conditions.
\end{proof}

Having discussed well-posedness of the Stokes and Darcy subproblems our main result concerning the coupled Darcy-Stokes problem \eqref{eq:darcy_stokes_weak} is given in Theorem 
\ref{lmm:DSwellposedness}. We remark that given two well-posed subproblems the coupled 
system could be analyzed with the framework of \cite{paper1}. Here we provide 
a standalone proof.

\begin{theorem} \label{lmm:DSwellposedness}
  Let $\Omega_f, \Omega_p$ be as defined in 
  \Cref{normalderivativeeps} and Assumption \ref{assumption:stokesbound}. Further let
  \[
  	V_f = \smuH,\, 
	Q_f = \ismuL,\,  
	Q_p = \sKH,\,  
	X_\Gamma = \DSmulte.
\]
  Then  
the operator $\AA$ in \eqref{coeff:darcy:stokes} is an isomorphism mapping 
$W$ to its dual space $W'$ such that 
$\|\AA\|_{\mathcal{L}(W,W')} \le C$ and 
$\|\AA^{-1}\|_{\mathcal{L}(W',W)} \le \frac 1C$, 
where $C$ is independent of $\mu$, $K$, and $D$.
\end{theorem}

\begin{proof}[Proof of Theorem \ref{lmm:DSwellposedness}]
We aim to apply Brezzi theory \cite{brezzi1974existence} to the Darcy-Stokes operator \eqref{coeff:darcy:stokes} 
in the abstract form \eqref{eq:brezzi}.
To this end let $V=V_f\times Q_p$ and $Q=Q_f\times X_{\Gamma}$ where for brevity 
$X_f = \DSmultf$, $X_p = \DSmultpe$ and we let the operators $A$, $B$ be defined in 
terms of bilinear forms from \eqref{eq:dsnaive} as
\[
\begin{aligned}
 (A(\uf, p_p), (\vf, q_p)) &=  \mu  (\nabla \uf, \nabla \vf)_{\Omega_f} + D (\uf \cdot \mathbf{\tau}, \vf \cdot \mathbf{\tau})_\Gamma  + K(\nabla p_p, \nabla q_p)_{\Omega_p}, \\  
(B(\uf, p_p), (q_f, w)) &=  (\nabla \cdot \uf, q_f)_{\Gamma} + (T_{n}\uf, w)_{\Gamma} - K(\Tgne p_p, w)_{\Gamma}.
\end{aligned}
\]

We proceed to verify the Brezzi conditions.
% We need to show:
%  \begin{description}
%  \item[\textbf{i)}] $a$ is bounded
%  \item[\textbf{ii)}] $b$ is bounded
%  \item[\textbf{iii)}] $a$ is coercive
%  \item[\textbf{iv)}] $b$ satisfies the inf-sup condition
%    $$\sup\limits_{(\uf, p_p) \in V} \frac {(\nabla \cdot \uf, q_f)_{\Gamma} + (T_{n}\uf, w)_{\Gamma} - K(\Tgne p_p, w)_{\Gamma}} {\|(\uf, p_p)\|_{V}} \geq \beta \|(q_f, w)\|_{Q}$$
%  \end{description}
    Note that by assumption $\semi{\partial\Omega_{i, D}}>0$, $i=p, f$ so that by Poincare 
    inequality on both subdomains $A$ is coercive. For boundedness of $A$ observe that 
    $\mu  (\nabla \uf, \nabla \vf) + D (\uf \cdot \mathbf{\tau}, \vf \cdot \mathbf{\tau})_\Gamma < \|\uf\|_{V_f} \|\vf\|_{V_f}$ by Cauchy Schwarz inequality. Moreover, following Lemma \ref{lmm:primaldarcybound}, we have
    $K(\nabla p_p, \nabla q_p) \leq \|p_p\|_{Q_p}\|q_p\|_{Q_p} $. Combining the two and applying the Cauchy-Schwarz inequality, 
    \[
    (A(\uf, p_p), (\vf, q_p)) \leq \|\uf\|_{V_f} \|\vf\|_{V_f}+  \|p_p\|_{Q_p}\|q_p\|_{Q_p} \leq \|(\uf, p_p)\|_V \|(\vf, q_p)\|_V.
    \]% First, we prove \textbf{i)}. By the Cauchy-Schwarz inequality, $\mu  (\nabla \uf, \nabla \vf)_{\Omega_f} + D (\uf \cdot \tG, \vf \cdot \tG)_\Gamma \leq \|\uf\|_{V_f} \|\vf\|_{V_f}$. Similarly, $K(\nabla p_p, \nabla q_p)_{\Omega_p} \leq K(p_p, q_p)_{\Omega_p} + K(\nabla p_p, \nabla q_p)_{\Omega_p}\leq \|p_p\|_{Q_p}\|q_p\|_{Q_p} $. Combining the two and applying Cauchy-Schwarz, $$a((\uf, p_p), (\vf, q_p)) \leq \|\uf\|_{V_f} \|\vf\|_{V_f}+  \|p_p\|_{Q_p}\|q_p\|_{Q_p} \leq \|(\uf, p_p)\|_V \|(\vf, q_p)\|_V$$
  
    To show boundedness of $B$ we recall that 
    $\|\nabla \cdot \uf\|_{Q^{\dual}_f} = \|\nabla \cdot \uf\|_{\sqrt{\mu}L^2(\Omega_f)} \leq C \|\uf\|_{V_f}  $, where $C$ depends on dimensionality of $\Omega_f$. Further, by the trace inequality $\|T_{n}\uf\|_{X^{\dual}_f} = \|T_{n}\uf\|_{\sqrt{\mu} H^{\half}(\Gamma)} \leq \|T_{n}\| \|\uf\|_{V_f}$, and by Assumption \ref{normalderivativeeps} $\|K\Tgne p_p\|_{X_p^{\dual}} = \|K\Tgne p_p\|_{\frac 1 {\sqrt K}L^2(\Gamma)} = \|\Tgne p_p\|_{{\sqrt K}L^2(\Gamma)} \leq \|\Tgne\| \|p_p\|_{Q_p} $. Hence, per definition of dual norms,
  \[
  (\nabla \cdot \uf, q_f)_{\Gamma} + (T_{n}\uf, w)_{\Gamma} \leq \|\nabla \cdot \uf\|_{Q^{\dual}_f} \|q_f\|_{Q_f} + \|T_{n}\uf\|_{X^{\dual}_f} \|w\|_{X_f} \leq \max(1, C)\|\uf\|_{V_f} \left ( \|q_f\|_{Q_f} + \|T_{n}\|\|w\|_{X_f} \right )\]
  and
  \[
  K(\Tgne p_p, w)_{\Gamma} \leq \|K\Tgne p_p\|_{X_p^{\dual}} \|w\|_{X_p} \leq \|\Tgne\| \|p_p\|_{Q_p} \|w\|_{X_p}.
  \]
  Combining the two we show boundedness of $B$
  \[
  (B(\uf, p_p), (q_f, w)) \leq  2\max (1, C, \|T_n\|, \|\Tgne\|\|) (\uf, p_p)\|_v \|(q_f, w)\|_Q.
  \]

  %Next, we show that the operators $a_S(\uf, \vf) = \mu (\nabla \uf, \nabla \vf) + D (\uf \cdot \tG, \vf %\cdot \tG)_{\Gamma}$ and $a_D(p_p, q_p) = K (\nabla p_p, \nabla q_p)$ are coercive on $V_f, Q_p$. By %the Poincar{\' e} inequality, there exists a constant $C_1$ so that $\|p\|_{L_2(\Omega_p)} \leq %C_1\|\nabla p\|_{L_2(\Omega_p)}$ and $\|\uf\|_{L_2(\Omega_f)} \leq C_1\|\nabla \uf\|_{L_2(\Omega_f)}$. %But then $\|p_p\|^2_{\sKH} \leq C_2 K \|\nabla p_p\|^2_{L^2(\Omega_p)} = C_2 a_D(p_p, q_p)$, so $a_D$ %is coercive on all of $\sKH$, and similarly $\|\uf\|^2_{\smuH} = \leq C_2 \sqrt{\mu}\|\nabla %%\uf\|^2_{L^2(\Omega_f)} +  D(\uf \cdot \tG, \vf \cdot \tG)_{\Gamma} \leq C_2 a_S(\uf, \vf)$, so $a_S$ %is coercive on all of $\smuH$. Therefore $a = a_S + a_D$ is coercive on $V_f \times Q_p$, proving %\textbf{iii)}.

  Finally, we turn to the inf-sup condition. Let $\Rszi{Q_f}, \Rszi{X_f}, \Rszi{X_p}$ be the inverse Riesz maps of their respective spaces, so that $R_Vu = (u, \cdot)_V \in V^{\dual}$. Let $(q_f, w) \in Q_f \times X$ be arbitrary. We first define two extensions by using the two subproblems. Recalling the notation of \Cref{assumption:stokesbound}, let $(\uf^*, p_f^*, \lambda*)$ be the solution of
    \begin{align*}
      a_S((\uf^*, p_f^*, \lambda^*), (\vf', q'_f, w')) =& \hspace{1mm} (\Rszi{Q_f}q_f, q'_f) + (\Rszi{X_f}w, w')_{\Gamma} \text{ for all } (\vf', q'_f, w') \in V_S
    \end{align*}
    Per assumption, there is a constant $C_f$ so that we have the bound
  %  \mk{What is with all the primes here? Also the stokes bound as written in the Assumption 
   % concerns norm of the solution triple while here only the velocity estimate is used} \karl{The primes are included because we have already used the symbols $q_f, w$ as the things we'll be defining $\uf^*, p_p^*$ in terms of, so we need notation for the test functions of the auxiliary subproblem we're solving to define $\uf^*$. I'm very open to better notation, but have no good ideas. As concerns the bound, this is true, we only actually need the bound on the velocity estimate. Is this a problem?}
    \begin{equation}
      \label{eq:stokesextbound}
      \|\uf^*\|_{V_f} \leq C_f \left ( \|\Rszi{Q_f}q_f\|^2_{\sqrt{\mu}L^2(\Omega_f)} + \|\Rszi{X_f} w\|^2_{\sqrt{\mu} H^{\half}(\Gamma)}  \right )^{\frac 1 2} = C_f \left ( \|q_f\|^2_{Q_f} + \|w\|^2_{X_f}  \right )^{\half}
    \end{equation}
    where the right equality follows from the fact that ${\sqrt{\mu}L^2(\Omega_f)} = Q^{\dual}_f$, ${\sqrt{\mu} H^{\half}(\Gamma)} = X_f^{\dual}$ and that the Riesz map is an isometry. Similarly, let $p_p^*, \lambda_2^*$ be the solution of
    $$K(\nabla p^*_p, \nabla q'_p)_{\Omega_p} + K(\lambda_2^*, \Tgne q'_p)_{\Gamma} + K(\Tgne p^*_p, w')_{\Gamma} = (\Rszi{X_p}w, w')_{\Gamma}\quad(q'_p, w') \in \sKH \times \DSmultpe.$$
    By \Cref{{lmm:primaldarcybound}}, we then have the bound
    \begin{equation}
\|p^*_p\|_{Q_p} \leq C_p\|\Rszi{X_p}w\|_{\frac 1 {\sqrt{K}} L^2(\Gamma)} = C_p\|w\|_{X_p}  \label{eq:darcyextbound}
\end{equation}

for a constant $C_p$. Observe now that by our definitions of $\uf^*, p^*_p$,  
    \begin{align*}
      &(\nabla \cdot \uf^*, q_f)_{\Gamma} + (T_{n}\uf^*, w)_{\Gamma} - K(\Tgne p^*_p, w)_{\Gamma} = (\Rszi{Q_f}q_f, q_f)_{\Gamma} + (\Rszi{X_f}w, w)_{\Gamma} - (\Rszi{X_p}, w)_{\Gamma} \\
      =& \|q_f\|^2_{Q_f} + \|w\|^2_{X_f} + \|w\|^2_{X_p} = \|q_f\|^2 + \|w\|^2_X.
    \end{align*}
    Using \eqref{eq:stokesextbound}, \eqref{eq:darcyextbound}
    $$\|(\uf^*, p^*_p)\|_{V} = \left (\|\uf^*\|^2_{V_f} + \|p^*_p)\|^2_{Q_p}\right )^{\frac 1 2} \leq C \left (\|q_f\|^2_{Q_f} + \|w\|^2_{X_f} + \|w\|^2_{X_p}\right )^{\frac 1 2} = C \left (\|q_f\|^2_{Q_f} + \|w\|^2_{X}\right )^{\frac 1 2}, $$
    where $C = \max(C_f, C_p)$. Putting this together, we can prove the inf-sup condition:
    \begin{align*}
      \sup\limits_{(\uf, p_p) \in V} & \frac {(\nabla \cdot \uf, q_f)_{\Gamma} + (T_{n}\uf, w)_{\Gamma} - K(\Tgne p_p, w)_{\Gamma}} {\|(\uf, p_p)\|_{V}}  \geq \frac {(\nabla \cdot \uf^*, q_f)_{\Gamma} + (T_{n}\uf^*, w)_{\Gamma} - K(\Tgne p^*_p, w)_{\Gamma}} {\|(\uf^*, p^*_p)\|_{V}} \\
           \geq&  \frac 1 C \frac {\|q_f\|^2 + \|w\|^2_X} {\left (\|q_f\|^2_{Q_f} + \|w\|^2_{X}\right )^{\frac 1 2}} = \frac 1 C \|(q_f, w)\|_{Q}.
    \end{align*}
    Hence, the inf-sup condition holds with $\beta = \frac 1 C$. By Brezzi theory, \Cref{lmm:DSwellposedness} follows, and the problem is well-posed. As in the argument of \cite{paper1}, we note that due to our use of parameter weighted spaces, all constants are in fact independent of the problem parameters, and the operator preconditioner is therefore robust to parameter variations. 

\end{proof}

Using operator preconditioning and Theorem \ref{lmm:DSwellposedness} a suitable
preconditioner for the primal Darcy-Stokes problem \eqref{eq:darcy_stokes_weak} 
is a Riesz map with respect to the inner product of $W$ in \eqref{eq:DS_spaces}, 
that is, the operator
\begin{equation}\label{eq:B_darcy_stokes}
\BB =    \left( \begin{array}{cccc} 
     -\mu \Delta + D \tG^{\prime}\tG &     \\ 
      & K \Delta   &   \\ \hline
      &  &  \frac{1}{\mu}I  &  \\
      & &    &  -\mu\left ( I + \Delta \right )_{\Gamma}^{-1/2} + \frac K h I  \\ 
     \end{array} \right)^{-1}. 
 \end{equation}

We remark that all of the components of the preconditioner can be realized in
an efficient, order optimal manner with multilevel schemes. In particular,
the only non-standard component here is the multilevel scheme for the
fractional operator which, however, has been established in \cite{baerland2018multigrid}.

\begin{example}[Robust Darcy-Stokes preconditioning]\label{ex:darcy_stokes}
  We consider the setup from Example \ref{ex:prelim} while using 
  the operator \eqref{eq:B_darcy_stokes} as preconditioner. As before, the leading blocks
  of the preconditioner are realized using single algebraic multigrid $V$-cycle.
  The multiplier block is then assembled using the eigenvalue decomposition and its
  inverse is computed by a direct solver.

  The obtained iteration and condition numbers are plotted in Figure \ref{fig:darcystokes}. 
  It can be seen that both quantities are bounded in mesh size $N$ as well 
  as the physical parameters $\mu$, $\kappa$ and $\alpha_{\text{BJS}}$.

  \begin{center}
    \begin{figure}[H]
      \includegraphics[width=0.49\linewidth]{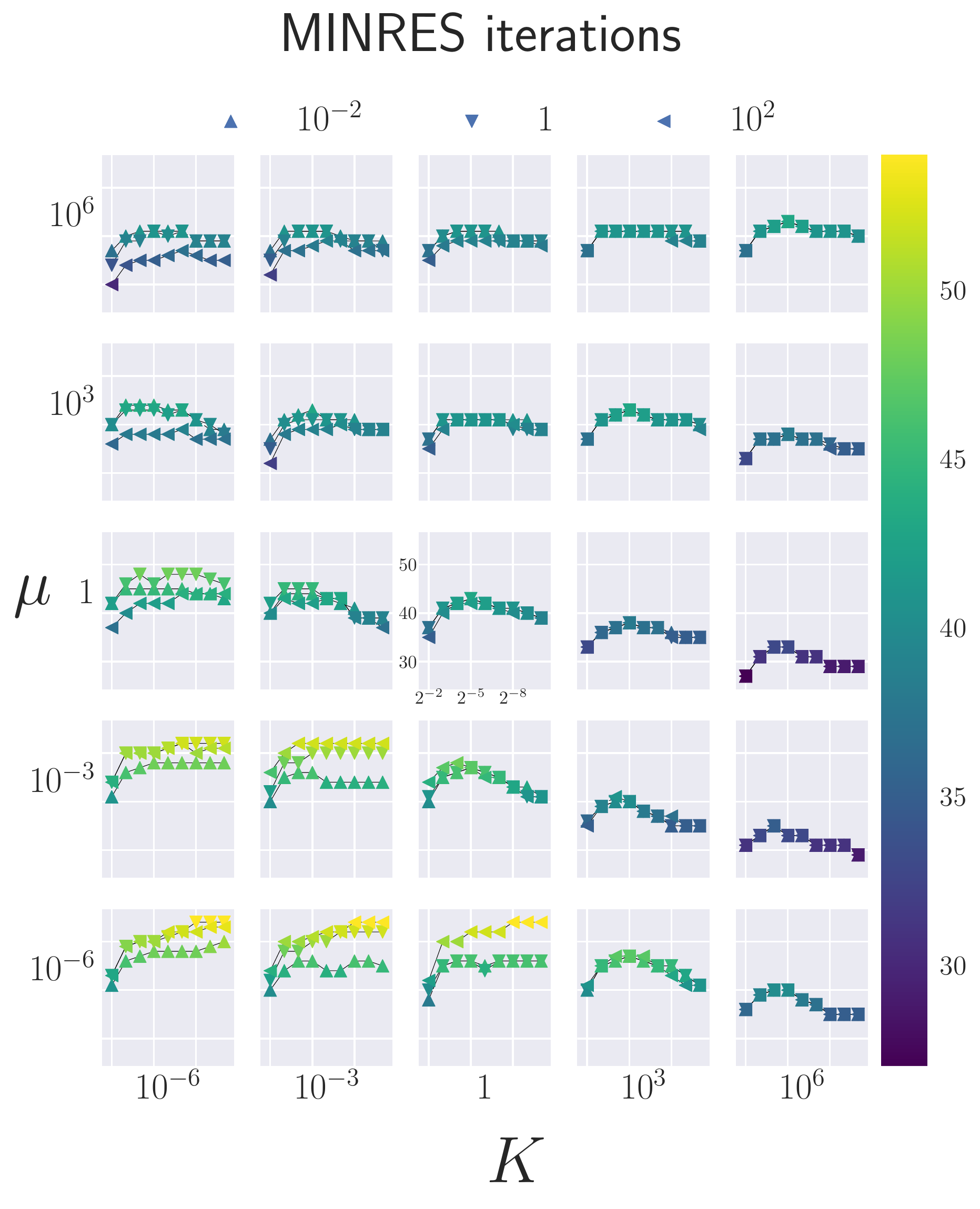}
      \includegraphics[width=0.49\linewidth]{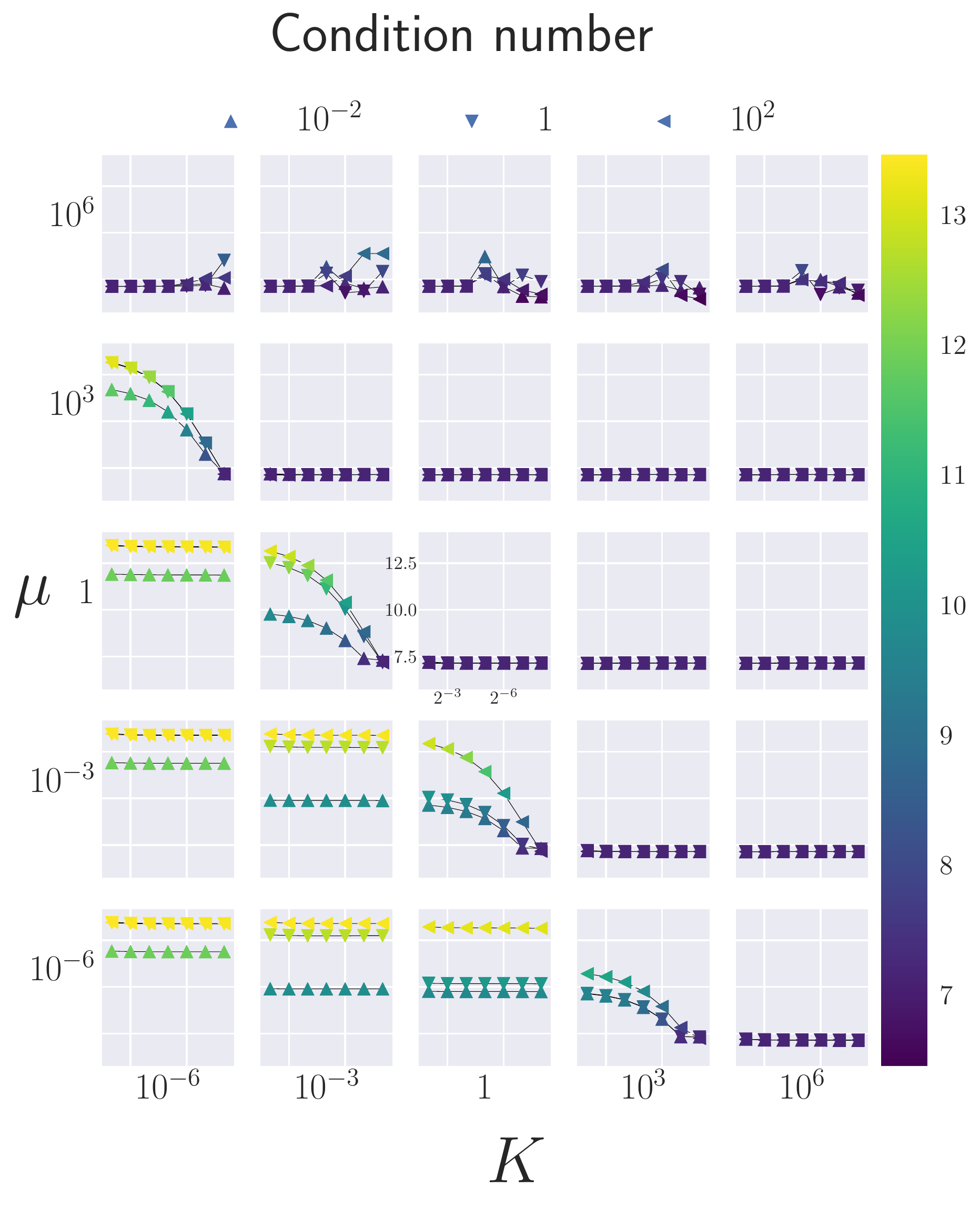}
      
    \caption{
             Mesh refinement vs. iteration counts (left) and condition numbers (right) for 
             Example \ref{ex:darcy_stokes} using the preconditioner \eqref{eq:B_darcy_stokes}.
         All subplots share $x$- and $y$-axes. 
         For fixed $\mu$, $K$ the 
         x-axis range in the iterations subplot extends from $h=2^{-2}$ to $h=2^{-10}$. In conditioning plots the range is 
         from $h=2^{-2}$ to $h=2^{-8}$.  
     The value of $\alpha_{\text{BJS}}$ is indicated by the line marker. 
     Triangles on top of each other look like squares.  
 \label{fig:darcystokes}
    }
    \vspace{5pt}    
\end{figure}
\end{center}

%   \begin{center}  
%   \begin{table}
%     \centering

%     \input{darcystokestable}
%     \caption{Iteration counts for the 
%  Darcy-Stokes Example \ref{ex:prelim}
%       using weighted preconditioner \eqref{eq:B_darcy_stokes}.
%     }
%         \label{tab:darcystokes}
%         \end{table}
% \end{center}
\end{example}

% probably unneeded since we treat it in the text

%% \section*{Appendix: Brezzi theory}
%% \label{sec:brezzi}
%% For completeness, we recall the classical Brezzi theory.

%% \begin{theorem}[Brezzi]
%%   Suppose $V, Q$ are two function spaces, and $a: V \times V \to \mathbb{R}$, $b : V \times Q \to \mathbb{R}$ are given bilinear forms.
%%   Suppose further that there exists constants $A, B, \alpha, \beta$ such that:
%%   \begin{description}
%%   \item[$a, b$ are bounded] $a(u, v) \leq A \|u\|_V\|v\|_V$ and $b(u, p) \leq B \|u\|_V\|p\|_Q$ for all $u, v \in V$, $p \in Q$
%%   \item[$a$ is coercive on $\ker b$] $a(u, u) \geq \alpha \|u\|^2_V$ for all $u$ such that $b(u, \cdot) = 0\in Q^*$
%%   \item[$b$ has the inf-sup property] $\inf_{p \in Q} \sup_{u \in V} \frac {b(u, p)}{\|p\|_Q \|u\|_V} \geq \beta >  0$
%%   \end{description}
%%   Then, given any $f \in V^*, g \in Q^*$, we can find unique $u, p \in V \times Q$ such that
%%   \begin{alignat*}{2}
%% a(u, v)& + b(v, p) &&= f(v) \\ 
%%  b(u, q)& &&= g(q)
%% \end{alignat*}
%% for all $v, q \in V \times Q$. The operator mapping $(f, g)$ to the solution $(u, p)$ is continuous, with operator norm bounded by $\alpha, \beta, A, B$. 
%% \end{theorem}

%\section{Babu{\v s}ka Neumann problem on general meshes}
%
%\kam{Here we do also the exact same tests on different meshes based on Example \ref{ex:bab_mass}.}  %
%%
%
%
%\label{sec:bab_nonuni}
%\include{babuska}

% \section{Discussion of \Cref{normalderivativeeps}} \label{sec:discussion}
\begin{remark}
Below we consider the validity of \Cref{normalderivativeeps} in a continuous and discrete setting. 
Clearly, in a continuous setting it is easy to find a function that violates the assumption. Consider the case where $\epsilon \ll h$ while 
$\Omega$ and $\Gamma$ are both unit sized. Further, let 
$u \in H^1(\Omega_p)$ be a function which is zero in $\Omega \backslash \Gamma_\epsilon$ and has  a gradient of 1 in $\Gamma_{\epsilon}$. 
Recalling our definition of the operator $\Tgne : H^1(\Omega_p) \rightarrow L^2(\Gamma_\epsilon)$ by $$\int_{\Gamma} \Tgne u \cdot w \, ds = \frac 1 {\epsilon} \int_{\Gamma_{\epsilon}} \nabla u \cdot \nn_{\Gamma_{\epsilon}} E_{\epsilon}w \, dx,  $$ we see that $\Tgne u \in L^2(\Gamma_\epsilon)$ is the unit constant function whereas $\|u\|_1 \approx \sqrt{\epsilon}$. Hence $\|\Tgne\| \geq \frac {\|\Tgne u\|_{L^2(\Gamma_\epsilon)}} {\|u\|_{H^1(\Omega_p)}} \approx \frac 1 {\sqrt \epsilon}$, which is very large for small $\epsilon$. Clearly, this function violates \Cref{normalderivativeeps}.

The above construction of a function that violates the assumption is however clearly not relevant in our discrete setting as these functions are 
below the resolution of our finite element mesh. Indeed, in our numerical experiments, we use discrete subspaces of $H^1(\Omega_p)$, so that any function whose gradient is nonzero on $\Gamma$ also has nonzero gradient at distance $h$ from $\Gamma$. This means that if $\epsilon$ is chosen smaller than $h$, functions like $u$ above which are zero immediately outside of $\Gamma_{\epsilon}$ are not admissible.  

For a relevant finite element function $u_h$, constructed as above, i.e.,  such that $u_h$ is 
zero everywhere except having a gradient of 1 on the finite elements with facets on $\Gamma$, assuming that $\epsilon \ll h$,   
we have $\frac {\|\Tgne u_h\|_{L^2(\Gamma_\epsilon)}} {\|u_h\|_{H^1(\Omega_p)}} \approx \frac {1} {\sqrt h}$. 
Indeed this estimate corresponds to the scaling shown in the Examples \ref{ex:L2_trace} and \ref{ex:bab_uni}.  
\end{remark}

\bibliography{multiphysics}

\begin{thebibliography}{10}
\expandafter\ifx\csname url\endcsname\relax
  \def\url#1{\texttt{#1}}\fi
\expandafter\ifx\csname urlprefix\endcsname\relax\def\urlprefix{URL }\fi
\expandafter\ifx\csname href\endcsname\relax
  \def\href#1#2{#2} \def\path#1{#1}\fi

\bibitem{paper1}
K.~E. Holter, M.~Kuchta, K.-A. Mardal, Robust preconditioning of monolithically
  coupled multiphysics problems.

\bibitem{arbogast2007computational}
T.~Arbogast, D.~S. Brunson, A computational method for approximating a
  {D}arcy--{S}tokes system governing a vuggy porous medium, Computational
  geosciences 11~(3) (2007) 207--218 (2007).

\bibitem{johnny2012family}
G.~Johnny, N.~Michael, A family of nonconforming elements for the {B}rinkman
  problem, IMA Journal of Numerical Analysis 32~(4) (2012) 1484--1508 (2012).

\bibitem{karper2009unified}
T.~Karper, K.-A. Mardal, R.~Winther, Unified finite element discretizations of
  coupled {D}arcy--{S}tokes flow, Numerical Methods for Partial Differential
  Equations: An International Journal 25~(2) (2009) 311--326 (2009).

\bibitem{mardal2002robust}
K.~A. Mardal, X.-C. Tai, R.~Winther, A robust finite element method for
  {D}arcy--{S}tokes flow, SIAM Journal on Numerical Analysis 40~(5) (2002)
  1605--1631 (2002).

\bibitem{zhang2009low}
S.~Zhang, X.~Xie, Y.~Chen, Low order nonconforming rectangular finite element
  methods for {D}arcy-{S}tokes problems, Journal of Computational Mathematics
  (2009) 400--424 (2009).

\bibitem{burman2007unified}
E.~Burman, P.~Hansbo, A unified stabilized method for {S}tokes’ and {D}arcy's
  equations, Journal of Computational and Applied Mathematics 198~(1) (2007)
  35--51 (2007).

\bibitem{feng2010stabilized}
M.-f. Feng, R.-s. Qi, R.~Zhu, B.-t. Ju, Stabilized {C}rouzeix-{R}aviart element
  for the coupled {S}tokes and {D}arcy problem, Applied Mathematics and
  Mechanics 31~(3) (2010) 393--404 (2010).

\bibitem{xie2008uniformly}
X.~Xie, J.~Xu, G.~Xue, Uniformly-stable finite element methods for
  {D}arcy-{S}tokes-{B}rinkman models, Journal of Computational Mathematics
  (2008) 437--455 (2008).

\bibitem{layton2002coupling}
W.~J. Layton, F.~Schieweck, I.~Yotov, Coupling fluid flow with porous media
  flow, SIAM Journal on Numerical Analysis 40~(6) (2002) 2195--2218 (2002).

\bibitem{galvis2007non}
J.~Galvis, M.~Sarkis, Non-matching mortar discretization analysis for the
  coupling {S}tokes-{D}arcy equations, Electron. Trans. Numer. Anal 26~(20)
  (2007) 07 (2007).

\bibitem{riviere2005locally}
B.~Rivi{\`e}re, I.~Yotov, Locally conservative coupling of {S}tokes and {D}arcy
  flows, SIAM Journal on Numerical Analysis 42~(5) (2005) 1959--1977 (2005).

\bibitem{gatica2008conforming}
G.~N. Gatica, S.~Meddahi, R.~Oyarz{\'u}a, A conforming mixed finite-element
  method for the coupling of fluid flow with porous media flow, IMA Journal of
  Numerical Analysis 29~(1) (2008) 86--108 (2008).

\bibitem{discacciati2009navier}
M.~Discacciati, A.~Quarteroni, {N}avier-{S}tokes/{D}arcy coupling: modeling,
  analysis, and numerical approximation (2009).

\bibitem{discacciati2003analysis}
M.~Discacciati, A.~Quarteroni, Analysis of a domain decomposition method for
  the coupling of {S}tokes and {D}arcy equations, in: Numerical mathematics and
  advanced applications, Springer, 2003, pp. 3--20 (2003).

\bibitem{cai2009preconditioning}
M.~Cai, M.~Mu, J.~Xu, Preconditioning techniques for a mixed {S}tokes/{D}arcy
  model in porous media applications, Journal of computational and applied
  mathematics 233~(2) (2009) 346--355 (2009).

\bibitem{discacciati2007robin}
M.~Discacciati, A.~Quarteroni, A.~Valli, Robin--{R}obin domain decomposition
  methods for the {S}tokes--{D}arcy coupling, SIAM Journal on Numerical
  Analysis 45~(3) (2007) 1246--1268 (2007).

\bibitem{ding1996proof}
Z.~Ding, A proof of the trace theorem of {S}obolev spaces on {L}ipschitz
  domains, Proceedings of the American Mathematical Society 124~(2) (1996)
  591--600 (1996).

\bibitem{kuchta2016preconditioners}
M.~Kuchta, M.~Nordaas, J.~C. Verschaeve, M.~Mortensen, K.-A. Mardal,
  Preconditioners for saddle point systems with trace constraints coupling 2d
  and 1d domains, SIAM Journal on Scientific Computing 38~(6) (2016) B962--B987
  (2016).

\bibitem{hughes2005isogeometric}
T.~J. Hughes, J.~A. Cottrell, Y.~Bazilevs, Isogeometric analysis: {CAD}, finite
  elements, {NURBS}, exact geometry and mesh refinement, Computer methods in
  applied mechanics and engineering 194~(39-41) (2005) 4135--4195 (2005).

\bibitem{nilssen2001robust}
T.~Nilssen, X.-C. Tai, R.~Winther, A robust nonconforming $h^2$-element,
  Mathematics of Computation 70~(234) (2001) 489--505 (2001).

\bibitem{mikelic2000interface}
A.~Mikelic, W.~J{\"a}ger, On the interface boundary condition of {B}eavers,
  {J}oseph, and {S}affman, SIAM Journal on Applied Mathematics 60~(4) (2000)
  1111--1127 (2000).

\bibitem{mardal2011preconditioning}
K.-A. Mardal, R.~Winther, Preconditioning discretizations of systems of partial
  differential equations, Numerical Linear Algebra with Applications 18~(1)
  (2011) 1--40 (2011).

\bibitem{fenics}
A.~Logg, K.~Mardal, G.~Wells, Automated Solution of Differential Equations by
  the Finite Element Method: The FEniCS Book, Lecture Notes in Computational
  Science and Engineering, Springer Berlin Heidelberg, 2012 (2012).

\bibitem{fenics_ii}
M.~Kuchta, Assembly of multiscale linear {PDE} operators, arXiv preprint
  arXiv:1912.09319 (2019).

\bibitem{hypre}
R.~D. Falgout, U.~M. Yang, hypre: A library of high performance
  preconditioners, in: International Conference on Computational Science,
  Springer, 2002, pp. 632--641 (2002).

\bibitem{petsc}
S.~Balay, S.~Abhyankar, M.~F. Adams, J.~Brown, P.~Brune, K.~Buschelman,
  L.~Dalcin, A.~Dener, V.~Eijkhout, W.~D. Gropp, D.~Karpeyev, D.~Kaushik, M.~G.
  Knepley, D.~A. May, L.~C. McInnes, R.~T. Mills, T.~Munson, K.~Rupp, P.~Sanan,
  B.~F. Smith, S.~Zampini, H.~Zhang, H.~Zhang,
  \href{https://www.mcs.anl.gov/petsc}{{PETS}c {W}eb page},
  \url{https://www.mcs.anl.gov/petsc} (2019).
\newline\urlprefix\url{https://www.mcs.anl.gov/petsc}

\bibitem{Hernandez:2005:SSF}
V.~Hernandez, J.~E. Roman, V.~Vidal, {SLEPc}: A scalable and flexible toolkit
  for the solution of eigenvalue problems, {ACM} Trans. Math. Software 31~(3)
  (2005) 351--362 (2005).

\bibitem{holter2017sub}
K.~E. Holter, M.~Kuchta, K.-A. Mardal, Sub-voxel perfusion modeling in terms of
  coupled 3d-1d problem, in: European Conference on Numerical Mathematics and
  Advanced Applications, Springer, 2017, pp. 35--47 (2017).

\bibitem{babuvska1973finite}
I.~Babu{\v{s}}ka, The finite element method with {L}agrangian multipliers,
  Numerische Mathematik 20~(3) (1973) 179--192 (1973).

\bibitem{brezzi1974existence}
F.~Brezzi, On the existence, uniqueness and approximation of saddle-point
  problems arising from {L}agrangian multipliers, Revue fran{\c{c}}aise
  d'automatique, informatique, recherche op{\'e}rationnelle. Analyse
  num{\'e}rique 8~(2) (1974) 129--151 (1974).

\bibitem{baerland2018multigrid}
T.~B{\ae}rland, M.~Kuchta, K.-A. Mardal, Multigrid methods for discrete
  fractional {S}obolev spaces, SIAM Journal on Scientific Computing 41~(2)
  (2019) A948--A972 (2019).

\end{thebibliography}

\end{document}